\documentclass[11pt]{amsart}
\usepackage{amssymb}
\usepackage{amsmath}
\usepackage{amsthm}
\usepackage{mathrsfs}
\usepackage{faktor}
\usepackage{tikz}
\usepackage{thmtools}
\usepackage{mathtools}

\usepackage{hyperref}

\usetikzlibrary{arrows}
\usetikzlibrary{shapes.geometric}
\usetikzlibrary{decorations.pathreplacing,decorations.markings}
\usepackage[utf8]{inputenc}
\usepackage{enumitem}
\usepackage{wasysym}

\usepackage[capitalise]{cleveref}

\newtheorem{theorem}{Theorem}[section]
\newtheorem{lem}[theorem]{Lemma}
\newtheorem{proposition}[theorem]{Proposition}
\newtheorem{cor}[theorem]{Corollary}
\newtheorem{qu}[theorem]{Question}

\theoremstyle{definition}
\newtheorem{dfn}[theorem]{Definition}

\newtheorem{rmk}[theorem]{Remark}

\numberwithin{theorem}{section}

\DeclareMathOperator{\Mon}{Mon}

\DeclareMathOperator{\BS}{BS}

\DeclareMathOperator{\Fix}{Fix}

\DeclareMathOperator{\expsum}{expsum}
\DeclareMathOperator{\Aut}{Aut}

\newcommand{\Z}{\mathbb{Z}}
\newcommand{\N}{\mathbb{Z}_{\geq 0}}

\title[Undecidable Diophantine problems]{Undecidable Diophantine problems in \\ generalisations of one-relator groups}

\thanks{The research of the authors was supported by the EPSRC Fellowship grant EP/V032003/1 ‘Algorithmic, topological and geometric aspects of infinite groups, monoids and inverse semigroups’. }

\begin{document}

\maketitle

\vspace{-4mm}

\begin{center}
 ROBERT D. GRAY
 \footnote{School of Engineering, Mathematics, and Physics, 
University of East Anglia, Norwich NR4 7TJ, England.
 Email \texttt{Robert.D.Gray@uea.ac.uk}.
 }
 and
 ALEX LEVINE
 \footnote{School of Engineering, Mathematics, and Physics, University of 
East Anglia, Norwich NR4 7TJ, England.
 Email \texttt{A.Levine@uea.ac.uk}.
 }
\end{center}

\vspace{-4mm}

\begin{abstract}
Motivated by the 
open problem of whether all one-relator groups have decidable Diophantine problem,
in this paper we prove a collection of undecidability results about the Diophantine problem for several families of groups that are close to one-relator groups in various ways. We prove that there is a generalised Baumslag--Solitar group with an undecidable Diophantine problem.  Using our example we show there is a group with an undecidable Diophantine problem that is quasi-isometric to a one-relator group. Also, we prove that there is a one-relator product of cyclic groups with an undecidable Diophantine problem. In addition, we show that there there is a one-relator group $G$, with a single fixed finite rank free subgroup $H$, such that the Diophantine problem for $G$ with $H$-constraints is undecidable. The related open question of whether there is a free-by-cyclic group with undecidable Diophantine problem is also discussed, and we prove that there is a free-by-free group of the form  $F_3 \rtimes F_2$ with an undecidable Diophantine problem.
\end{abstract}

\section{Introduction}
It is a longstanding open problem whether all one-relator groups have decidable
conjugacy problem; see \cite[Problem (O5)]{baumslag2002open}.
See \cite[Section~1.8]{CFMarcoOneRelatorSurvey} for background on this problem. As explained there, while it remains open in general, for some classes of one-relator groups it is known to be decidable e.g. cyclically pinched one-relator groups \cite{Lipschutz1966}, one-relator groups with torsion by the B. B. Newman Spelling Theorem 
and more generally for any hyperbolic one-relator group \cite{Gromov1987}, and free-by-cyclic one-relator groups \cite{Bogopolski2006}. 
Other work on the conjugacy problem in one-relator groups, and in the more general class of one-relator products, includes 
\cite{Duncan1991, Minasyan2013, Pride2008}.
In recent related work in \cite{bridson2025conjugacy} the complexity of the conjugacy problem, and the conjugacy search problem, have been studied in detail for the two generator one-relator hydra groups.
We note that it is also not know whether every finitely generated subgroup of a one-relator group has decidable conjugacy problem; see 
\cite[Problem 1.8.2]{CFMarcoOneRelatorSurvey}.

The conjugacy problem for one-relator groups is a special case of another open problem which asks whether every one-relator group has decidable Diophantine problem -- the problem of whether a system of equations over the group has a solution. This problem is known to be decidable for free semigroups and free groups \cite{Makanin_semigroups, Makanin_eqns_free_group, Razborov_thesis} and more generally for hyperbolic groups \cite{dahmani_guirardel, Sela}. 
The Diophantine problem has also been shown to be decidable for right-angled Artin groups \cite{DiekertMuscholl}, virtually direct products of hyperbolic groups \cite{CiobanuHoltRees}, and various torsion-free relatively hyperbolic groups \cite{dahmani2009existential}. 
Other recent important work in this area includes 
\cite{CasalsRuiz2011} which gives algorithms for constructing Makanin--Razborov diagrams to encode the solution sets of equations over RAAGs, 
and 
the more recent paper \cite{CasalsRuiz2024} on the elementary theory of graph products of groups.
Not all $3$-manifold groups have decidable Diophantine problem. Indeed, 
several infinite families of Seifert 3-manifold groups with undecidable Diophantine problem are given in \cite{GrayLevineArxiv251207690}. 
It remains an interesting open problem to classify those 3-manifold groups for which the Diophantine problem is decidable.  

Among one-relator groups, the Diophantine problem for Baumslag--Solitar groups has received attention recently in the literature, with a particular focus on the groups $BS(1,n)$ for $n \in \mathbb{N}$ with $n \geq 2$. The Diophantine problem for quadratic equations in $BS(1, n)$ was shown to be decidable in
\cite{kharlampovich2020diophantine}.
In related work in 
\cite{Duncan2023}
EDT0L systems are used to solve some equations in this class of Baumslag--Solitar groups. Note that since $BS(1,n)$ is solvable but not virtually abelian it follows from \cite{Ersov72, noskov1984elementary, romanovskiui1981elementary}
that $BS(1,n)$ hence has an undecidable first-order theory. Other work related to the Diophantine problem in $BS(1,n)$ includes 
\cite{Mandel2023}
and \cite{Mandel2023b}. 
In \cite{CFOneRelatorDiophantine} it is observed that results from 
may be applied to deduce that the unimodular Baumslag--Solitar groups $BS(m,m)$ and $BS(m,-m)$ have decidable Diophantine problem. In that same paper it is shown that all torus knot groups have decidable Diophantine problem and more generally so do all two-generator one-relator groups with defining relator of the form $a^m = b^n$. In \cite{CiobanuGrayLevineInPrep} it is shown that this result generalises in two different ways. Firstly, every cyclically pinched one-relator group has decidable Diophantine problem, and secondly every one-relator group with non-trivial centre has decidable Diophantine problem.   

The Diophantine problem is also open for one-relator monoids. In fact, even the word problem remains open that that class; see \cite{CFOneRelatorMonoidsSuvey}. Some work on the Diophantine problem for one-relator monoids includes 
\cite{Kharlampovich2019} and  
\cite{Garreta2021}
where it is shown that if this problem were decidable it would provide a positive answer to a longstanding open problem in theoretical computer science about decidability of word equations with length constraints. 

In this paper we will prove a collection of undecidability results about the Diophantine problem for several families of groups that are close to one-relator groups in various ways. As well as being of interest in their own right, these results may be viewed as evidence 
suggesting there may exist 
one-relator groups with undecidable Diophantine problem.  
We now summarise the main new results that we shall prove in this paper, and explain how they relate to other work in the literature. For any undefined notions we refer the reader to Section~\ref{sec:prelim}.

As explained above, the Diophantine problem is open for Baumslag--Solitar groups $BS(m,n)$. The first result we prove in this paper is the following.  

\noindent \textbf{Theorem~A.} 
\emph{
There is a generalised Baumslag--Solitar group with an undecidable Diophantine problem.} 

We shall then show how the example we use for Theorem~A may be used to prove the following. 

\noindent \textbf{Theorem~B.} 
\emph{
There is a group $G$ with an undecidable Diophantine problem such that $G$ is quasi-isometric to a one-relator group.}

Related to this, it is unknown whether decidability of the Diophantine problem is a quasi-isometry invariant among finitely presented groups; and it is is open whether it is preserved under taking finite index subgroups or extensions.

\noindent \textbf{Theorem~C.} 
\emph{ There is a one-relator product of cyclic groups 
with an undecidable Diophantine problem. }

We note that both the word problem and the conjugacy problem are open in general for one-relator products of cyclic groups. 

\noindent \textbf{Theorem~D.} 
\emph{ 
There is a one-relator group $G$,  
with a single fixed finite rank free subgroup $H$, 
such that the Diophantine problem for $G$ with $H$-constraints is undecidable.} 

Related to this, the subgroup membership problem is open for one-relator groups;
see \cite[Problem 1.8.8]{CFMarcoOneRelatorSurvey}.

We also consider the related open problem of whether there is a free-by-cyclic group $F_n \rtimes \Z$ group with an undecidable Diophantine problem. We prove:  

\noindent \textbf{Theorem~E.} 
\emph{ 
There is a group of the form $F_3 \rtimes F_2$ with an undecidable Diophantine problem.} 

As far as the authors are aware, the conjugacy problem is open for groups of the form $F_3 \rtimes F_2$. Note that any group of this form has just 6 defining relations and the conjugacy problem remains open for $k$-relator groups for all $1 \leq k \leq 10$; see \cite[Section 1.8]{CFMarcoOneRelatorSurvey}. 
The conjugacy problem is however known to be decidable for free-by-cyclic groups $F_n \rtimes \Z$ (see \cite{Bogopolski2006}), a result which has recently been extended to arbitrary ascending HNN-extensions \cite{Logan}. 
Also, in 
\cite{BogopolskiMartinoVentura} 
it is proved that the conjugacy problem is decidable for all groups of the form $F_2 \rtimes F_m$, and an example is given of a group of the form $F_3 \rtimes F_{14}$ with undecidable conjugacy problem. 

For the rest of the introduction we shall give more details of the results above and explain how they relate to other work in the literature.   
We will establish Theorem~A and Theorem~B by proving the following result about generalised Baumslag--Solitar groups which, as is standard, we abbreviate to GBS groups.   

\textbf{Theorem 3.4.}
\emph{
Let \(p, q \in \mathbb{Z}\) be distinct prime numbers. Then the GBS group \[G = \langle a, s, t \mid sas^{-1} = a^p, tat^{-1} = a^q \rangle \]
has an undecidable Diophantine problem.} 

Combining our result with results of 
Whyte \cite{Whyte01} on quasi-isometry classes of generalised Baumslag--Solitar groups we obtain the following corollary.  

\textbf{Corollary 3.7.}
\emph{There is a group \(G\) that is quasi-isometric to all Baumslag-Solitar groups \(\BS(m, n)\), where \(1 < m < n\), such that \(G\) has an undecidable Diophantine problem.}

It would be of interest to classify the GBS groups with decidable Diophantine problem although this problem is likely to be difficult since it is still open for Baumslag--Solitar groups. 

Several results for one-relator groups have been successfully extended to
one-relator products. Often these results are concerned with one-relator
products of locally indicable groups; see e.g.  \cite{Howie84}. More recently in
\cite{JaikinZapirainLintonSanchezPeralta} it is shown that a torsion-free
one-relator product of locally indicable groups is coherent provided that both
factor groups are coherent. 
A generalisation of the conjugacy problem, called the genus problem, has been studied for one-relator products of locally indicable groups in \cite{Duncan1991}.
More background on one-relator products the introduction to the paper
\cite{DuncanArye} and in the survey articles \cite{DuncanHowie92} and \cite{FineRosenberger}. 

Our main result about these groups concerns one-relator products of cyclic groups. Of course this class generalises one-relator groups which are all quotients of a free product of finitely many copies of $\Z$ by a single defining relation.
One-relator products of cyclic groups have been extensively studied. 
Indeed,  
Chapter 6 and 7 of the book 
\cite{FineRosenberger}
are devoted to the topic of one-relator products of cyclic groups. 
Important work on one-relator products of cyclic groups includes 
\cite{
Brodskii84,
CampbellHeggieRobertsonThomas,
DuncanHowie93,
FineHowieRosenberger88,
Howie81,
HowieFourthPowers}.
One-relator products of cyclic groups are of interest in part due to their connections with topics like generalized triangle groups \cite{BaumslagMorganShalen}.

While, as explained in the introduction of the paper \cite{DuncanArye}, the word problem is known to be decidable in certain one-relator products of cyclic groups, it remains open in general and that same is true of the conjugacy problem for one-relator products of cyclic groups.  
Our result Theorem~C will be proved by establishing the following result. 

\textbf{Theorem 4.3.}
\emph{There is a one-relator product of cyclic groups the form $(\Z \ast C_2) / \langle \langle r \rangle \rangle$ with undecidable Diophantine problem. Specifically the group
\[
  K_2 = \langle x, t \mid
x x^{x^t} = x^{x^t} x,  
\quad 
  t^2=1
  \rangle
\]
has an undecidable Diophantine problem.}

Another longstanding open algorithmic problem for one-relator groups is whether they have decidable subgroup membership problem (also sometimes called the generalised word problem); see \cite[Problem 1.8.8]{CFMarcoOneRelatorSurvey}. A decision problem that mixes both the Diophantine problem and the subgroup membership problem is the Diophantine problem with subgroup constraints. Given a finitely generated group $G$ and a finitely generated subgroup $H$ we say \emph{$G$ has decidable Diophantine problem with $H$-constraints} if there as an algorithm that takes any finite system of equations over $G$ together with a finite list of constraints of the form $X_1, \ldots, X_k \in H$ for variables $X_i$ and decides whether or not it has a solution. It is immediate from the definition that if a group $G$ has Diophantine problem with $H$-constraints then the membership problem for $H$ in $G$ must be decidable. We will prove below that while both Diophantine problem and the subgroup membership problem remain open in general for one-relator groups, when we combine them in this way then the problem does become undecidable. 
We shall prove Theorem~D by establishing the following result.  

\textbf{Theorem~5.1.}
\emph{The one-relator group $G = \langle c, q \mid cc^q = c^qc \rangle$ has a single fixed subgroup where $H$ is isomorphic to the free group of rank $6$ and such that the Diophantine problem for $G$ with $H$-constraints is undecidable.} 

The group $G = \langle c, q \mid cc^q = c^qc  \rangle$ in the statement of this theorem is the well-known  non-subgroup separable one-relator group of Burns, Karrass and Solitar \cite{Burns1987}.
It is also a particular instance of a hydra group in the sense of \cite{HydraGroups}. 
We do not know whether or not this group  
has a decidable Diophantine problem.
It is know (see e.g. \cite[Section~2]{Burns1987}) that this one-relator group 
$G = \langle c, q \mid cc^q = c^qc  \rangle$
is a free-by-cyclic group of the form $F_2 \rtimes \Z$ with presentation     
\[
  \langle a, b, s \mid sas^{-1} = ab, sbs^{-1} = b\rangle.
\]
Now free-by-cyclic groups has decidable conjugacy problem \cite{Bogopolski2006}, but the Diophantine problem for free-by-cyclic groups $F_n \rtimes \Z$ is open. 
While that question remains open, we shall show that there are relatively simple free-by-free groups with undecidable Diophantine problem.  
Specifically we shall prove the following result, which establishes Theorem~E.    

\textbf{Theorem 6.5.}
\emph{There is a group of the form $F_3 \rtimes F_2$ with an undecidable Diophantine problem. In particular the group
\[
\langle a, b,c, s, t \mid sas^{-1} = ab, sbs^{-1} = b, scs^{-1} = c, tat^{-1} = ba, tbt^{-1} = b, tct^{-1} = c \rangle
\]
has an undecidable Diophantine problem.}

As mentioned above, as far as the authors are aware, the conjugacy problem is open for groups of the form 
$F_3 \rtimes F_2$. 

\section{Preliminaries}
\label{sec:prelim}

In this section we shall give some definitions and notation that will be needed in the rest of the article. We assume the reader has familiarity with standard notions from combinatorial and geometric group theory. For more details we refer the reader to 
\cite{lyndon1977combinatorial, magnus2004combinatorial} and \cite{delaHarpeBook}.

\subsection{Groups, submonoids and quasi-isometries}
If \(G\) is a group and
	\(X \subseteq G\), we write \(\Mon\langle X \rangle\) for the submonoid
	of \(G\) generated by \(X\). If \(g \in G\), then we define
	\(C_G(g) = \{h \in G \mid gh = hg\}\); this set is called the
	\emph{centraliser} of \(g\).
For a non-empty alphabet $A$ and subset $R$ of $F_A \times F_A$, where $F_A$ denotes the free group on $A$, we use $\langle A \mid R \rangle$ to denote the group defined by the presentation with generators $A$ and defining relations $R$.  
Given any letter $s \in A$ and word $u \in F_A$ we use  $\expsum_{s}(u) \in \mathbb{Z}$ to denote the \emph{exponent sum of $s$ in $u$}, that is, the total number of occurrences of $s$ in $u$ minus the to total number of occurrences of $s^{-1}$ in $u$.    

If $G$ and $H$ are finitely generated groups with finite generating sets $A$ and $B$, respectively, then $G$ and $H$ are \emph{quasi-isometric} if the Cayley graph $\Gamma(G,A)$ of $G$ with respect to $A$ is quasi-isometric to $\Gamma(H,B)$ as a metric space; see \cite[Chapter IV]{delaHarpeBook}. Whether or not two finitely generated groups are quasi-isometric is independent of choice of finite generating sets for the groups. A group is quasi-isometric to its finite index subgroups, so commensurable finitely generated groups are quasi-isometric, but the converse does not hold in general.  

\subsection{Equations in groups}

    \begin{dfn}
		Let \(G\) be a finitely generated group and \(\mathcal X\) be a
		finite set. An \emph{equation} in \(G\) is an element \(w \in G
		\ast F(\mathcal X)\), which we denote by \(w = 1\). A
		\emph{solution} to \(w = 1\) is a homomorphism \(\phi \colon G \ast
		F(\mathcal X) \to G\), such that \(\phi(w)  = 1\) and \(\phi (g)=
		g\) for all \(g \in G\).  The elements of \(\mathcal X\) are called
		the \emph{variables} of \(w = 1\). A \emph{(finite) system of
		equations} in \(G\) is a finite collection of equations in \(G\)
		using the same set of variables. A \emph{solution} to a system
		\(\mathcal E\) of equations is a homomorphism that is a solution to
		all equations in \(\mathcal E\). If \(\mathcal E\) is a system of
		equations, when writing out \(\mathcal E\), we wll use the logical
		operator \(\wedge\) to delimit the equations in the system; that is
		\(\mathcal E = (w_1 = 1) \wedge (w_2 = 1) \wedge \cdots \wedge (w_n
		= 1)\).

	    The \emph{Diophantine problem} in \(G\) is the decision question which asks if
	    there is an algorithm taking as input a system \(\mathcal E\) of equations in \(G\)
	    and outputting whether or not \(\mathcal E\) admits a solution. 
	\end{dfn}

		We will often abuse notation by considering a solution to an equation
		with \(n \in \mathbb{Z}_{>0}\) variables \(X_1, \ldots, X_n\) in a
		group \(G\) to be an \(n\)-tuple of elements \((g_1,  \ldots,  g_n)\).
		The corresponding homomorphism \(\phi\) can be recovered from such a
		tuple by defining \(\phi\) to be the unique homomorphism satisfying
		\(\phi(g)  = g\) for all \(g \in G\) and \(\phi(X_i) = g_i\) for all
		\(i\). We additionally abuse notation by considering equations of the
		form \(w_1 = w_2\). Such an equation is used to denote the equation
		\(w_1w_2^{-1} = 1\).

When showing various groups have undecidable Diophantine problems, a
	key method is studying which sets can be `defined' using systems of
	equations. We begin with a definition of solution sets, and then
	formally define equationally definable.
    \begin{dfn}
        Let \(G\) be a finitely generated group and \(\mathcal E\) be a
        system of equations in \(G\). Let \((X_1, \ldots, X_n)\) be
		some 
        (not necessarily all) variables occurring in
		\(\mathcal E\). We define the \emph{set of solutions} to
        \((X_1, \ldots, X_n)\) in \(\mathcal E\) to be set
        \[
        \{(g_1, \ldots, g_n) : \textrm{ there exist } h_1, \ldots, h_m
                \in G \textrm{ such that } (g_1, \ldots, g_n, h_1, \ldots, h_m)
\text{ is a solution to } \mathcal E\}.\]
	Let \(D \subseteq G^n\) for some $n$. We say \(D\) is \emph{equationally
	definable} if there is a system \(\mathcal E\) of equations in \(G\),
	and a collection \((X_1, \ldots, X_n)\) of (not necessarily all) variables
	in \(\mathcal E\), such that \(D\) is the set of solutions to
	\((X_1, \ldots, X_n)\) in \(\mathcal E\). 
	\end{dfn}

\section{A GBS group with undecidable Diophantine problem}
A \emph{generalised Baumslag-Solitar group} (abbreviated to GBS group) is the fundamental group of a finite graph of groups where all vertex and edge groups are infinite cyclic. 
GBS groups have arisen in the literature in the study of finitely generated groups of cohomological dimension 2 \cite{Kropholler1990}, JSJ decompositions \cite{Kropholler1993}, one-relator groups with centre \cite{Collins1977, McCool1991, Pietrowski1974} and mapping tori \cite{levitt2015generalized}.

	The group in question we will be studying here is the generalised
	Baumslag-Solitar group \(G = \langle a, s, t \mid sas^{-1} = a^p, tat^{-1}
	= a^q \rangle\), where \(p\) and \(q\) are distinct primes. 

The group \(G\) can be realised as a free product with amalgamation of two solvable Baumslag-Solitar groups over an infinite cyclic subgroup, or alternatively, an HNN-extension of a solvable Baumslag-Solitar group.  We note later that it is quasi-isometric to \(\BS(2, 3)\).

Before we can attempt to discuss equations in \(G\), we must first construct a normal form with which to work. In this case, we are using the standard normal form for free products with amalgamation; see e.g. \cite[Chapter IV]{lyndon1977combinatorial}.

	\begin{lem}
		\label{GBS-norm-form}
		Let \(G = \langle a, s, t \mid sas^{-1} = a^p, tat^{-1} = a^q
		\rangle\), where \(p\) and \(q\) are distinct primes. 
Define
				the following sets of freely reduced words:
				\begin{align*} 
T_s & =    
\{ s^{-i} a^j s^k \mid i,k \in \Z_{> 0}, \; j \in \{0, \ldots, p^i - 1\}, \; j \not\in p\Z \}  \\
& \cup \{ s^k \mid k \in \Z_{\geq 0} \} \\ 
& \cup \{ s^{-i} a^j \mid i \in \Z_{\geq 0}, \; j \in \{0, \ldots, p^i - 1\} \}, \\ 
T_t & =    
\{ t^{-i} a^j t^k \mid i,k \in \Z_{> 0}, \; j \in \{0, \ldots, q^i - 1\}, \; j \not\in q\Z \}  \\
& \cup \{ t^k \mid k \in \Z_{\geq 0} \} \\ 
& \cup \{ t^{-i} a^j \mid i \in \Z_{\geq 0}, \; j \in \{0, \ldots, q^i - 1\} \}.   
\end{align*} 

		Then 
		\[
			\{a^r x_1 y_1 \cdots x_n y_n \mid r \in \Z, n \in \Z_{\geq 0},
					x_1 \in T_s, x_2, \ldots, x_n \in T_s \setminus \{1\}, 
			y_1, \ldots, y_{n - 1} \in T_t \setminus \{1\}, y_n \in T_t \}
		\]
		is a normal form for \(G\).
	\end{lem}

	\begin{proof}
		We begin by showing that \(T_s\) is a right transversal for \(\langle
		a \rangle\) in the Baumslag-Solitar group \(B_s := \langle a , s \mid sas^{-1}
		= a^p\rangle\). From \cite[Section~2]{BurilloElder}, we have that \(B_s\) admits
		the normal form
		\[
				\{ s^{-i} a^j s^k \mid i, k \in \Z_{> 0}, j \in \Z \setminus p\Z \} 
				\cup \{s^{-i} a^j s^k \mid i, k \in \Z_{\geq 0}, ik = 0, j \in \Z\}.
		\]
		We first show that \(T_s\) contains a right transversal for \(\langle a
		\rangle\); that is, \(B_s = \langle a \rangle T_s\). Let \(s^{-i} a^j
		s^k \) be an arbitrary normal form word in \(B_s\). We can write \(j =
		d p^i + r\) for some \(d \in \Z\) and \(r \in \{0, \ldots, p^i - 1\}\).
		Then \(s^{-i} a^j s^k = s^{-i} a^{d p^i} a^r s^k = a^d s^{-i} a^{r}
		s^k\). In either of the two cases of the normal form for \(B_s\)
		that \(s^{-i} a^j s^k\) lies in, this new word \(s^{-i} a^r s^k\) is
		always in \(T_s\), as required.

		We next show that distinct elements of \(T_s\) lie in distinct right
		cosets of \(\langle a \rangle\), which will complete our proof of the
		fact that \(T_s\) is a right transversal. Let \(s^{-i_1} s^{j_1}
		s^{k_1}, s^{-i_2} a^{j_2} s^{k_2} \in T_s\) be such that \(s^{-i_1}
		a^{j_1} s^{k_1} (s^{-i_2} a^{j_2} s^{k_2})^{-1} \in \langle a
		\rangle\). Then
		\begin{align*}
				s^{-i_1} a^{j_1} s^{k_1} (s^{-i_2} a^{j_2} s^{k_2})^{-1}
				& = s^{-i_1} a^{j_1} s^{k_1} s^{-k_2} a^{-j_2} s^{i_2} \\
				& = s^{-i_1} a^{j_1} s^{k_1 - k_2} a^{-j_2} s^{i_2} \\
				& = \begin{cases}
						s^{-i_1} a^{j_1 - j_2 p^{k_1 - k_2}} s^{i_2 + k_1 - k_2}
						& k_1 - k_2 > 0 \\
						s^{-i_1 + k_1 - k_2} a^{j_1 p^{k_2 - k_1} - j_2}
						s^{i_2} & k_1 - k_2 < 0 \\
						s^{-i_1} a^{j_1 - j_2} s^{i_2} & k_1 - k_2 = 0.
				\end{cases}
		\end{align*}
		We next consider the various cases arising from the types of
		words in \(T_s\); the determining factor for our purposes
		being whether or not \(j_1\) and \(j_2\) equal \(0\).

		Case 1: \(j_1, j_2 > 0\). The definition of \(T_s\) thus implies
		that \(i_1, i_2 > 0\), \(j_1\in \{0, \ldots, p^{i_1} - 1\}
		\setminus p\Z\) and \(j_2 \in \{0, \ldots, p^{i_2} - 1\} \setminus
		p\Z\). Then if \(k_1 - k_2 > 0\), we have that \(i_2 + k_2 - k_2 > 0\),
		so in this case the above word \(s^{-i_1} a^{j_1 - j_2 p^{k_1 - k_2}}
		s^{i_2 + k_1 - k_2}\) is in normal form. As the set of normal forms for
		elements of \(\langle a \rangle\) is just \(\{a^i \mid i \in \Z\}\),
		it follows that \(s^{-i_1}a^{j_1 - j_2 p^{k_1 - k_2}}
		s^{i_2 + k_1 - k_2} \notin \langle a \rangle\). Thus \(s^{-i_1} a^{j_1}
		s^{k_1}\) and \(s^{-i_2} a^{j_2} s^{k_2}\) lie in different right
		cosets of \(\langle a \rangle\).

		Similarly, if \(k_1 - k_2 <
		0\), then the word \(s^{-i_1 + k_1 - k_2} a^{j_1 p^{k_2 - k_1} - j_2}
		s^{i_2}\) is also in normal form. As the normal form words for
		\(\langle a \rangle\) are \(\{a^i \mid i \in \Z\}\), neither of these
		elements lie in \(\langle a \rangle\). It remains to consider the case
		when \(k_1 - k_2 = 0\); that is, the word \(s^{-i_1} a^{j_1 - j_2}
		s^{i_2}\). If \(j_1 - j_2 \notin p \Z\), then \(s^{-i_1} a^{j_1 - j_2}
		s^{i_2}\) is in normal form and thus does not lie in \(\langle a
		\rangle\), so suppose \(j_1 - j_2 \in p\Z\). Then \(j_1 - j_2 = d
		p^\ell\), where \(d \in \Z \setminus p \Z\) and \(\ell \in \Z_{> 0}\).
		Since \(s^{-i_1} a^{j_1 - j_2} s^{i_2} \in \langle a \rangle\), we have
		that \(s^{-i_1} a^{j_1 - j_2} s^{i_2} / \langle \langle a \rangle
		\rangle_{B_s} = 1\). 
This therefore implies that \(i_1 = i_2\).
Next note that since \(j_1, j_2 \in \{ 1, \ldots,
		p^{i_1} - 1\}\), we have that \(-p^{i_1} < j_1 - j_2 < p^{i_1}\). Thus
		\(\ell < i_1\). Next note that \(s^{-i_1} a^{j_1 - j_2} s^{i_1}
		= s^{-i_1} a^{d p^\ell} s^{i_2} = s^{-i_1 + \ell} a^d s^{i_1 - \ell}\).
		This word is now in normal form, but the only case when this
		lies in \(\langle a \rangle\) is when \(d = 0\). This implies
		\(j_1 - j_2 = 0\), and so \(i_1 = i_2\), \(j_1 = j_2\) and
		\(k_1 = k_2\), and we have shown that the two elements of
		\(T_s\) we started with are equal.

		Case 2: \(j_1 = 0\) and \(j_2 \neq 0\). Then the definition of
		\(T_s\) tells us that \(i_1 = 0\) or \(k_1 = 0\). 
		Again, we consider when \(k_1 - k_2\) is positive, negative and zero.
		If \(k_1 - k_2 > 0\), then \(k_1 > 0\), and so \(i_1 = 0\). In this
		case our word is \(s^{-i_1} a^{j_1 - j_2 p^{k_1 - k_2}} s^{i_2 + k_1 - k_2}
		\equiv a^{-j_2 p^{k_1 - k_2}} s^{i_2 + k_1 - k_2}\), which is in normal
		form, and thus lies in \(\langle a \rangle\) if and only if
		\(i_2 + k_1 - k_2 = 0\). But this is impossible, since \(i_2, k_1 - k_2 >0\),
		and so we can conclude that \(k_1 - k_2 \leq 0\).

		If \(k_1 - k_2 < 0\), then our word is
		\(s^{-i_1 + k_1 - k_2} a^{j_1 p^{k_1 - k_2} - j_2} s^{i_2}
		\equiv s^{-i_1 + k_1 - k_2} a^{j_2} s^{i_2}\), which is
		in normal form since \(-i_1 + k_1 - k_2 < 0\) and \(j_2 \notin p \Z\).
		Thus our word lies in \(\langle a \rangle\) if and only if
		\(-i_1 + k_1 - k_2 = i_2 = 0\), which is impossible since
		\(-i_1 \leq 0\) and \(k_1 - k_2 < 0\). It follows that
		\(k_1 - k_2 \geq 0\), and when taken with the previous paragraph,
		we have \(k_1 - k_2 = 0\).

		Since \(k_1 - k_2 = 0\), our word is \(s^{-i_1} a^{j_1 - j_2}
		s^{i_2} \equiv s^{-i_1} a^{-j_2} s^{i_2}\), which is in normal
		form, since \(j_2 \notin p\Z\). This lies in \(\langle a_1
		\rangle\) if and only if \(i_1 = i_2 = 0\), but the definition
		of \(T_s\) tells us that \(i_2 = 0\) implies that \(j_2 = 0\),
		a contradiction. We can conclude that this case does not occur.

		Case 3: \(j_1 \neq 0\) and \(j_2 = 0\). This is analogous to
		Case 2.

		In this case \(s^{-i_1} a^{j_1} s^{k_1} \equiv s^{\ell_1}\)
		and \(s^{-i_2} a^{j_2} s^{k_2} \equiv s^{\ell_2}\) for some
		\(\ell_1, \ell_2 \in \Z\). Then
		\(s^{\ell_1} (s^{\ell_2})^{-1} = s^{\ell_1 - \ell_2}\),
		which is in normal form. It thus lies in \(\langle a \rangle\)
		if and only if \(\ell_1 = \ell_2\); that is, if and only if
		the two words we began with were equal.

		We have thus shown that no two elements of \(T_s\) lie in
		the same right coset of \(\langle a \rangle\), and so
		\(T_s\) is a right transversal for \(\langle a \rangle\)
		in \(B_s\).
		By symmetry, \(T_t\) is a right transversal for
		\(\langle a \rangle\) in \(\langle a, t \mid tat^{-1} = a^q\rangle\).

		Noting that \(G\) is the amalgamated free product of
		\(\langle a, s \mid sas^{-1} = a^p \rangle\) with \(\langle a, t
		\mid tat^{-1} = a^q \rangle\) across the subgroup \(\langle a\rangle\),
		\cite[Theorem 11.3]{BogopolskiBook} tells us that \(G\) has the normal form	
\[
\{a^r x_1 y_1 \cdots x_n y_n \mid r \in \Z, n \in \Z_{\geq 0},
					x_1 \in T_s, x_2, \ldots, x_n \in T_s \setminus \{1\}, 
			y_1, \ldots, y_{n - 1} \in T_t \setminus \{1\}, y_n \in T_t \}.
\]
\end{proof}

One example of a set that can be defined using equations is a centraliser.
	Centralisers have been frequently employed when showing Diophantine
	problems or Diophantine problems with constraints are undecidable, for
	example in \cite{BuchiSenger, CiobanuGarreta, duchin_liang_shapiro, ElliottLevine,  random_nilpotent_eqns, KharlampovichMiasnikov2025}, and they will play a key role in our proof too. 

\begin{lem}
		\label{GBS-cents}
		Let \(p, q \in \mathbb{Z}\) be distinct prime numbers. Define 
		\(G = \langle a, s, t \mid sas^{-1} = a^p, tat^{-1} = a^q \rangle\).
		Then the following hold in \(G\):
		\begin{enumerate}
			\item \(C_G(t) = \langle t \rangle\) and \(C_G(s) = \langle s \rangle\);
					\label{GBS-cent-s-t}
			\item For all \(i \in \Z_{\geq 0}\), \(C_G(st^i) = \langle st^i \rangle\).
					\label{GBS-cent-sti}
		\end{enumerate}
	\end{lem}

	\begin{proof}
			We begin by showing \eqref{GBS-cent-s-t} and \eqref{GBS-cent-sti} when \(i = 0\); that is we compute
		\(C_G(s)\) and \(C_G(t)\).  
Since the surjective homomorphism from $G$ to the free group $F_{s,t}$ given by mapping $a$ to $1$ does not map $s$ to $1$ it follows that    
\(s \in
		\langle a, s \rangle \setminus \langle \langle a \rangle \rangle_G\), and similarly  
\(t \in \langle a, t \rangle \setminus \langle \langle a \rangle
		\rangle_G\).
Thus
		\cite[Proposition 2.3]{BarkThesis} tells us that 
\(C_G(s) = C_{\langle a, s \rangle}(s)\) and
		\(C_G(t) = C_{\langle a, t \rangle}(t)\). These subgroups are solvable
		Baumslag-Solitar groups, and so it follows that we can apply
		\cite[Proposition 2.15]{BarkThesis}, noting that \(s a^k = a^{pk} s
		\neq a^k s\) to conclude that \(C_G(s) = C_{\langle a, s \rangle}(s) =
		\langle s \rangle\) and \(C_G(t) = C_{\langle a, t \rangle}(t) =
		\langle t \rangle\).

		For the remainder of \eqref{GBS-cent-sti}, suppose \(i > 0\).
		It is immediate that \(\langle st^i \rangle \subseteq
		C_G(st^i)\).  For the reverse containment, let \(g \coloneqq a^r x_1 y_1 \cdots x_n
		y_n \in C_G(st^i)\) be in normal form.

		Case 1: \(y_n \neq 1\) and \(x_1 \neq 1\). Since \(y_n \neq 1\), 
		\(a^r x_1 y_1 \cdots x_n y_n st^i\) is in normal form with
		\(s \in T_s \setminus \{1\}\) and \(t \in T_t\). In
		addition,
		\begin{align*}
				st^i g & = s t^i \cdot a^r x_1 y_1 \cdots x_n y_n \\ 
			& = a^{r pq^i} st^i x_1 y_1 \cdots x_n y_n. 
		\end{align*}
		Since \(x_1 \neq 1\), this is in normal form with \(s \in T_s\) and
		\(t^i \in T_t \setminus \{1\}\), and so
		\(rpq^i = r\), and so \(r = 0\). In addition,
		\(st^i x_1 y_1 \cdots x_n y_n = x_1 y_1 \cdots x_n y_n st^i\). This
		implies \(x_1 = s, y_1 = t^u, x_k = x_{k + 1}, y_k = y_{k + 1}\)
		for all \(k \in \{1, \ldots, n - 1\}\), \(x_n = s\) and \(y_n = t^i\).
		Thus \(s = x_1 = x_2 = \cdots = x_n\) and \(t^i = y_1 = \cdots 
		= y_n\), and so \(g= (st^i)^n \in
		\langle st^i \rangle\), as required.
			
		Case 2: \(y_n = 1\) and \(x_1 \neq 1\) and \(x_n \neq s^{-1}\). 
		Then \(g st^i\) has normal form
		\(a^r x_1 y_1 \cdots x_n st^i\), with \(x_n s \in T_s \setminus \{1\}\)
		and \(t^i \in T_t\). In addition,
		\begin{align*}
			st^i g & = st^i a^r x_1 y_1 \cdots x_n \\
				   & = a^{r pq^i} st^i x_1 y_1 \cdots x_n.
		\end{align*}
		Since \(x_1 \neq 1\), the above word is in normal form
		with \(s \in T_s\) and \(t^i \in T_t \setminus \{1\}\).
		Then \(t^i = 1\), a contradiction.

		Case 3: \(y_n = 1\), \(x_1 \neq 1\) and \(x_n = s^{-1}\).
		Then \(gst^i = a^r x_1 y_1 \cdots x_{n - 1} y_{n - 1} t^i\),
		which is in normal form, with \(y_{n - 1} t^i \in T_t\). In
		addition,
		\begin{align*}
				st^i g & = st^i a^r x_1 y_1 \cdots x_{n - 1} y_n s^{-1} \\
					   & = a^{rpq^i} st^i x_1 y_1 \cdots x_{n - 1} y_{n -1} s^{-1}.
		\end{align*}
		This is in normal form, with \(s \in T_s\), \(t^i \in T_t \setminus
		\{1\}\) and \(s^{-1} \in T_s\). However, the normal form for \(gst^i\)
		has \(2(n - 1) +1 = 2n - 1\) alternating terms, whilst the normal form
		for \(st^i g\) has \(2(n - 1) + 3 = 2n + 1\), a contradiction.

		Case 4: \(y_n \neq 1\) and \(x_1 = 1\). Then \(g st^i
		= a^r y_1 x_2 \cdots x_n y_n st^i\), which is a normal form word.
		In addition,
		\begin{align*}
			st^i g & = st^i a^r y_1 x_2 \cdots x_n y_n
				    = a^{rpq^i} st^i y_1 x_2 \cdots x_n y_n.
		\end{align*}
		This last word is a normal form word. Since the normal form words
		for \(st^i g\) and \(g st^i\) differ, they are distinct, a contradiction.
	\end{proof}

	We now move on from centralisers to more general equationally definable
	subsets. The latter three sets come
	close to showing that exponent-sum constraints in free monoids are
	`realisable' using the Diophantine problem in \(G\). Whilst we do not
	show that this is formally the case, the extent to which this
	`realisation' does hold is sufficent to allow us to use an analogue of
	the proof of undecidability of the Diophantine problem with exponent-sum
	constraints in a free monoid \cite{BuchiSenger}.

	\begin{lem}
		\label{GBS-defable}
		Let \(p, q \in \mathbb{Z}\) be distinct prime numbers.
		The following sets are equationally definable in
		\(G = \langle a, s, t \mid sas^{-1} = a^p, tat^{-1} = a^q \rangle\):
		\begin{enumerate}
				\item \(\Mon\langle s \rangle\) and \(\Mon\langle t \rangle\);
						\label{GBS-defable-monst}
				\item \(\{(g, t^i) \mid  i \in \Z_{\geq 0}, g \in \Mon\langle st^i
						\rangle\}\); 
						\label{GBS-defable-g-ti}
		\end{enumerate}
		In addition, if \(E\) is equationally definable
				in \(G\) and admits two coordinates that only take values in
				\(\Mon\langle s, t \rangle\), then the following three sets
				are equationally definable in \(G\), where the intersections
				are taken between the two coordinates in \(E\) taking
				values in \(\Mon \langle s, t \rangle\) and the given sets:
		\begin{enumerate}[resume]
				\item \(E \cap \{(u_1, u_2) \mid u_1, u_2 \in \Mon\langle s, t \rangle,
						\expsum_{s}(u_1) = \expsum_t(u_2)\}\); 
						\label{GBS-defable-expsum-s-t}
				\item \(E \cap \{(u_1, u_2) \mid u_1, u_2 \in \Mon\langle s, t \rangle,
						\expsum_{s}(u_1) = \expsum_s(u_2)\}\); 
						\label{GBS-defable-expsum-s-s}
				\item \(E \cap \{(u_1, u_2) \mid u_1, u_2 \in \Mon\langle s, t \rangle,
				\expsum_{t}(u_1) = \expsum_t(u_2)\}\). 
						\label{GBS-defable-expsum-t-t}
		\end{enumerate}
	\end{lem}

	\begin{proof}
			For \eqref{GBS-defable-monst} and \eqref{GBS-defable-g-ti}, using \cref{GBS-cents}, we will begin by showing
			that the set of solutions to
			\begin{equation}
					\label{eqn:GBS-Monsti}
					(X st^i = st^i X) \wedge (XaX^{-1} t^{-(i + 1)} at^{i + 1} = t^{-(i + 1)}
			at^{i + 1} XaX^{-1})
			\end{equation}
			is precisely \(\Mon\langle st^i \rangle\), for all \(i \in \Z_{\geq
			0}\). The fact that \(\Mon \langle s \rangle\) is equationally
			definable is the special case of this when \(i = 0\).

			Fix \(i \geq 0\), and let \(g \in 
			 C_G(st^i)\); that is, \(g\) is a solution to the first above equation.
			 By \cref{GBS-cents}(2), \(C_G(st^i) = \langle st^i \rangle\), and so
			 \(g = (st^i)^j\) for some \(j \in \Z\). We consider the cases when
			 \(j \geq 0\) and \(j < 0\) separately.

			 If \(j \geq 0\) (that is, \(g \in \Mon\langle st^i \rangle\)),
			 then \(gag^{-1} \in \langle a \rangle\). Further, \(t^{-(i + 1)} a
			 t^{i + 1} \cdot a = t^{-(i + 1)} a \cdot a^{p^{i + 1}} t^{i + 1} =
			 a \cdot t^{-(i + 1)} at^{i + 1}\), and so \(gag^{-1} t^{-(i + 1)}
			 a t^{i + 1} = t^{-(i + 1)} a t^{i + 1} gag^{-1}\). Thus
			 \((st^i)^j\) is a solution to the system Equation~\eqref{eqn:GBS-Monsti}
			 for all \(j \geq 0\).

	Next suppose that \(j <
			0\).  Then
			\begin{align*}
					gag^{-1} t^{-(i + 1)} a t^{i + 1} & = (st^i)^j a (st^i)^{-j} t^{-(i + 1)} a t^{i + 1} \\
										& = (st^i)^{j + 1} t^{-i} s^{-1} a s
										t^i (st^i)^{-j - 1} t^{-(i + 1)} a t^{i + 1}
										\\
										& = 
										\begin{cases}
										t^{-i} s^{-1} a s
										 t^{-1} a t^{i + 1} & j = -1 \\
										(st^i)^{j + 1} t^{-i} s^{-1} a s
										t^i (st^i)^{-j -2} s t^{-1} a t^{i + 1}
										& j < -1,
										\end{cases}
										\\
										t^{-(i + 1)} a t^{i + 1} gag^{-1}
										& = t^{-(i + 1)} a t^{i + 1}
										(st^i)^j a (st^i)^{-j} \\
										& = t^{-(i + 1)} a t^{i + 1}
										(st^i)^{j + 1} t^{-i} s^{-1} a s
								t^i (st^i)^{-j - 1} \\
										& = \begin{cases}
										t^{-(i + 1)} a t s^{-1} a s t^i & j = -1 \\
										t^{-(i + 1)} a t s^{-1}
										(st^i)^{j + 2} t^{-i} s^{-1} a s
								t^i (st^i)^{-j - 1} 
																		& j < -1.
										\end{cases}
			\end{align*}
			We first consider the case that \(j = -1\). We can see that
			\(t^{-i}, t^{-1} a t^{i + 1} \in T_t \setminus \{1\}\) and \(s^{-1}
			a s\in T_s \setminus \{1\}\) (using the definitions of \(T_s\) and
			\(T_t\) from \cref{GBS-norm-form}) and so \(t^{-i} s^{-1} a s
			t^{-1} a t^{i + 1}\) is in normal form. Similarly we have that
			\(t^{-(i + 1)} a t, \; t^i \in T_t \setminus \{1\}\) and
			\(s^{-1} a s \in T_s \setminus \{1\}\), and so
			\(t^{-(i + 1)} a t s^{-1} a st^i\) is also in normal form.
			Thus the group elements these words represent coincide
			if and only if these words are equal as words, which they do not,
			and so this case of \(j = -1\) contributes no solutions to
			the system Equation~\eqref{eqn:GBS-Monsti}.

			Similarly, if \(j < -1\), then \(s, s^{-1} as \in T_s \setminus
			\{1\}\) and \(t^i, t^{-i}, t^{-1} a t^{i + 1} \in T_t \setminus
			\{1\}\), and so the word \((st^i)^{j + 1} t^{-i} s^{-1} a s t^i
			(st^i)^{-j -2} s t^{-1} a t^{i + 1}\) is in the normal form from
			\cref{GBS-cents}. One can see similarly that \(t^{-(i + 1)} a t
			s^{-1} (st^i)^{j + 2} t^{-i} s^{-1} a s t^i (st^i)^{-j - 1}\) is in
			normal form, and so these elements coincide if and only if these
			words are equal. It is immediate that this never happens when \(j <
			0\), and so this case also does not contribute solutions to the
			system Equation~\eqref{eqn:GBS-Monsti}.

We have thus shown that the set of solutions to the equation
			Equation~\eqref{eqn:GBS-Monsti} is \(\{(st^i)^j \mid j \in \Z_{\geq 0}\} =
			\Mon \langle st^i \rangle\), and so \(\Mon\langle st^i \rangle\) is
			equationally definable in \(G\) for all \(i \in \Z_{\geq 0}\). A
			particular case of this is when \(i = 0\), which shows that
			\(\Mon \langle s \rangle\) is equationally definable. By
			symmetry, we obtain that \(\Mon \langle t \rangle\) is also
			equationally definable and we have shown \eqref{GBS-defable-monst}.

			To obtain the more
			general desired set for \eqref{GBS-defable-g-ti}, note that if \(\mathcal E\) is a
			system of equations in \(G\) with a variable \(Y\) admitting the
			set of solutions \(\Mon\langle t \rangle\), then we have shown
			above that the set of solutions to \(\mathcal E \wedge (XaX^{-1}
			t^{-1} Y^{-1} a Y t = t^{-1} Y^{-1} a Y t XaX^{-1}) \wedge (sY X =
			XsY)\) is precisely the desired set stated in \eqref{GBS-defable-g-ti}.

For \eqref{GBS-defable-expsum-s-s} (\eqref{GBS-defable-expsum-t-t}
			is symmetric), consider the set of solutions to \((X, Y)\) in \( X
			U a(XU)^{-1} = Y a (Y)^{-1}\), with the constraints that \((X, Y)
			\in E\) and \(U \in \Mon \langle t \rangle\).  So let \((x, y) \in
			E \subseteq (\Mon \langle s, t \rangle)^2\), and \(i \in \Z_{\geq
			0}\). Then \(xt^i a (xt^i)^{-1} = a^{p^{\expsum_s(x)}q^{i +
			\expsum_t(x)}}\) and \(y a y^{-1} = a^{p^{\expsum_s(y)}
			q^{\expsum_t(y)}}\). Since \(p\) and \(q\) are distinct primes,
			this holds if and only if \(\expsum_s(x) = \expsum_s(y)\) and \(i +
			\expsum_t(s) = \expsum_t(y)\).  Thus the set of \(x\), \(y\) and
			\(i\) for which this is satisfied is
			\[
				\{(x, y, t^i) \in E \times \Mon\langle t^i \rangle \mid
				\expsum_s(x) = \expsum_s(y) \text{ and } \expsum_t(x) =
		\expsum_t(y) + i\},
			\]
			and so this is the set of solutions to \((X, Y, U)\) in this system.
			Thus if we simply consider the solutions to \((X, y)\) in this
			system, we obtain the projection of the above set to the first two
			coordinates, which is:
			\[
				\{(x, y) \in E \times \Mon\langle t^i \rangle \mid
		\expsum_s(x) = \expsum_s(y)\},
			\]
	which is the desired set.
			
	For \eqref{GBS-defable-expsum-s-t}, let \(\mathcal E\) be a system of equations in \(G\) with
			a variable \(U\), such that the set of solutions to \(U\) is \(\Mon
			\langle st \rangle\), which exists by \eqref{GBS-defable-g-ti}.  Let \(X\) and \(Y\) be
			additional variables with the constraints that \(X, Y \in \Mon
			\langle s, t \rangle\). We can then use \eqref{GBS-defable-expsum-s-s} and \eqref{GBS-defable-expsum-t-t} to insist
			that \(\expsum_s(X) = \expsum_s(U)\) and \(\expsum_t(Y) = 
			\expsum_t(U)\). Since \(U \in \Mon\langle st \rangle\),
			\(\expsum_s(U) = \expsum_t(U)\), and so this system of
			equations will have the desired set of solutions to
			\((X, Y)\).
	\end{proof}

	Having completed our setup, we can move on to showing undecidability. As
	with most proofs of undecidability of Diophantine problems in groups, we
	simulate a solution to Hilbert's tenth problem of the natural numbers using
	the Diophantine problem for \(G\). Since the former is undecidable
	\cite{Matijasevic}, this will complete the proof.
	\begin{theorem}
			\label{GBS-thm}
		Let \(p, q \in \mathbb{Z}\) be distinct prime numbers. Then the GBS group
		\[
			G = \langle a, s, t \mid sas^{-1} = a^p, tat^{-1} = a^q \rangle
		\]
		has undecidable Diophantine problem.
	\end{theorem}

	\begin{proof}
		We encode Hilbert's Tenth problem over the semiring of natural numbers
		into \(G\).  Undecidability of the Diophantine problem for \(G\) will
		then follow from the undecidability of Hilbert's tenth problem over the
		semiring of natural numbers. The fact that \(\Mon\langle s \rangle
		\cong \N\) (as a monoid) provides the copy of the natural numbers in
		which we simulate Hilbert's tenth problem.  \cref{GBS-defable}\eqref{GBS-defable-monst} tells us
		that this submonoid is equationally definable. To encode addition, note
		that the equation \(XY =Z\) in \(G\) with the (equationally definable,
		by \cref{GBS-defable}\eqref{GBS-defable-monst}) constraints \(X, Y, Z \in \langle s \rangle\)
		has set of solutions \(\{(s^i, s^j, s^{i + j}) \mid i, j \in \N\}\).

		We now encode multiplication into the Diophantine problem of \(G\);
		that is, we show that \(\{(s^i, s^j, s^{ij}) \mid i, j \in \N\}\) is
		equationally definable in \(G\). First note that
		\cref{GBS-defable}\eqref{GBS-defable-g-ti} tells us that \(\{(t^i,
		(st^i)^j) \mid i, j \in \N\}\) is equationally definable. Let
		\(S_{st}\), \(S_s\) and \(S_t\) denote the three sets from parts
		\cref{GBS-defable}\eqref{GBS-defable-expsum-s-t},
		\eqref{GBS-defable-expsum-s-s} and \eqref{GBS-defable-expsum-t-t} that
		form equationally definable subsets of \(G\) when intersected with
		equationally definable subsets of \(G\) that lie in \(\Mon\langle s, t
		\rangle\). 

		We then add further equations
		to obtain the following system, noting that whenever the constraints
		\(S_{s}\), \(S_t\) and \(S_{st}\) are applied, the two variables
		they are applied to have already been forced to lie in
\(\Mon \langle s, t \rangle\) by previous equations.
		\[
			(U, V) \in \{(t^i, (st^i)^j) \mid i, j \in \N\} \wedge
			X, Y, Z \in \Mon\langle s \rangle \wedge
			((X, U) \in S_{st}) \wedge
			((Y, V) \in S_s) \wedge
			((Z, V) \in S_t).
		\]
		By construction, the set of
		solutions to \((U, V, X, Y, Z)\) is
		\[
			\{(t^i, (st^i)^j, s^i, s^j, s^{ij}) \mid i, j \in \N\}
		\]
		which shows that \(\{(s^i, s^j, s^{ij}) \mid i, j \in \N\}\) is
		equationally definable in \(G\). When combined with the fact that
		addition is also equationally definable in \(G\), this gives that
		Hilbert's tenth problem over the semiring of natural numbers can be
		encoded in the Diophantine problem of \(G\). Since the former is
		undecidable (\cite{Matijasevic}; see also
		\cite[Corollary 2.14]{ElliottLevine} for the exact format of
		\cite{Matijasevic} we use), the undecidability of the latter follows.
	\end{proof}

	Since \(G\) can be realised as an HNN-extension of a solvable
	Baumslag-Solitar group over \(\Z\), which is an HNN-extension of \(\Z\)
	over \(\Z\), we can thus conclude that HNN-extensions over \(\Z\)
	do not appear amongst the properties which preserve the decidability
	of the Diophantine problem. The class of groups with decidable Diophantine
	problem is known to be preserved under HNN-extensions and free products
	with amalgamation over a finite group \cite{LohreySenizergues}.

	\begin{cor}
		The class of groups which have a decidable Diophantine Problem is not
		closed under HNN-extensions over \(\Z\).
    \end{cor}

	\begin{proof}
		In light of \cref{GBS-thm}, it suffices to show that \(G = \langle a, s, t \mid sas^{-1} = a^p, tat^{-1} = a^q \rangle\) can be built up from \(\Z\) (which has a decidable
		Diophantine problem) using HNN-extensions over \(\Z\). This can be seen by
		first noting that \(\BS(1, p) = \langle a, s, \mid sas^{-1} = a^p \rangle\) is
		an HNN-extension of \(\Z\) over \(\Z\), and \(G\) is an HNN-extension of
		\(\BS(1, p)\) over \(\Z\).
	\end{proof}

It is interesting to contrast the results in this section with related results on the conjugacy problem for HNN extensions and amalgamated free products. For example in \cite[Theorem 5.10]{WeissThesis} it is proved that if $G$ is a finite graph of groups and such that all vertex groups are finitely generated free or free abelian and all edge groups cyclic then $G$ has decidable conjugacy problem. The GBS group constructed in this section shows that the same result does not hold when the conjugacy problem is replaced by the Diophantine problem.  

	We do not show that the class of groups which have a decidable Diophantine problem
	is not closed under free products amalgamted over \(\Z\). If solvable Baumslag-Solitar
	groups were to have a decidable Diophantine problem, this would provide a
	counter-example, although perhaps there is a more straightforward way of proving this.

	\begin{qu}
		Is the class of groups with a decidable Diophantine problem closed under
		free products amalgmated over \(\Z\)?
	\end{qu}

We next show how by combining our result \cref{GBS-thm} with a result of Whyte \cite{Whyte01} we can show that various Baumslag-Solitar groups are quasi-isometric to a group with undecidable Diophantine problem. 

\begin{cor}
There is a group \(G\) that is quasi-isometric to all Baumslag-Solitar groups \(\BS(m, n)\), where \(1 < m < n\), such that \(G\) has an undecidable Diophantine problem.
	\end{cor}

	\begin{proof}
		By \cref{GBS-thm}, it suffices to show that \(G = \langle a, s, t \mid
		sas^{-1} = a^p, tat^{-1} = a^q \rangle\) is quasi-isometric to \(\BS(m,
		n)\) for all \(1 < m < n\). We first note that as \(G\) has an
		undecidable Diophantine problem (\cref{GBS-thm}), it cannot contain a
		finite index subgroup of the form \(F_n \times \Z\), where \(F_n\) is a
		finite rank free group, as this contradicts
		\cite[Theorem 3.1]{CiobanuHoltRees}. Furthermore, \(\Z^2 \leq G / [G, G]\),
		and so \(G\) is not isomorphic to \(\BS(1, n)\) for any \(n \in \Z \setminus \{0, 1\}\),
		and \(G\) is not abelian, and thus is not isomorphic to \(\BS(1, 1)\). We
		can now apply \cite[Theorem 0.1]{Whyte01} to conclude that \(G\) is
		quasi-isometric to \(\BS(2, 3)\), and \cite[Corollary 0.2]{Whyte01} to
		conclude that \(\BS(m, n)\) is quasi-isometric to \(\BS(2, 3)\) for
		all \(1 < m < n\), as required.
	\end{proof}

Another consequence of what we have proved is the following. 

\begin{cor} 
At least one of the following statements is true: 
\begin{itemize} 
\item 
The property of having a decidable Diophantine problem is not a quasi-isometry invariant of finitely presented groups; 
\item All Baumslag-Solitar groups \(\BS(m, n)\), where \(1 < m < n\) have undecidable Diophantine problems. 
  \end{itemize} 
  \end{cor}

We suspect that the class of finitely presented groups which have a decidable Diophantine problem is not closed under quasi-isometry. 
Decidability of the conjugacy problem is known not to be a quasi-isometry invariant of finitely presented groups. In fact, it is known that decidability of the conjugacy problem is neither preserved by passing finite index subgroups nor extensions of index two; see \cite{MillerBook}. It is currently an open question whether the property of having decidable Diophantine problem is preserved by finite index subgroups or extensions.   

In contrast, it is well-known that for finitely presented groups having a decidable word problem is a quasi-isometry invariant; see \cite{delaHarpeBook}. For finitely generated but not finitely presented groups this is no longer true; e.g. the first Grigorchuk group is quasi-isometric to a group with undecidable word problem; see \cite{mathoverflow}. 

There are other interesting algorithmic questions that remain open for GBS groups e.g. the open isomorphism problem; see  \cite{AscariCasalsRuizKazachkov} for recent work on this problem. 
In related work, in the paper \cite[Lemma 3.2]{CasalsRuizetal21} they prove the following. Let $n \geq m \geq 1$ and $\mathrm{gcd}(m,n)=d$, $m=dp$, $n = dq$ such that $\mathrm{gcd}(p,q)=1$. Then the Baumslag--Solitar group $BS(m,n)$ contains a finite index subgroup isomorphic to the group
\[G_{p,q}^d = \langle a, t_1, \ldots, t_d \mid t_i^{-1} a^p t_i = a^q, i = 1, \ldots, d \rangle. \]
Given that decidability of the Diophantine problem remains open for the groups $BS(m,n)$ it would be interesting to study the same question for the groups $G_{p,q}^d$.   

\section{A one-relator product of cyclic groups}
In this section we consider Diophantine problems for one-relator products of cyclic groups. More background on one-relator products of groups may be found e.g. in the introduction to the paper \cite{DuncanArye}. In general both the word problem and conjugacy problems remain open for one-relator products of cyclic groups, although they have been shown to be decidable in certain cases. For example, it is known that if $A$ and $B$ are locally indicable groups then $Q = A \ast B/ \langle \langle r \rangle \rangle$ has a decidable word problem in the case that $r = s^m$, where $m > 1$. Also, if $A$ and $B$ are arbitrary (not necessarily locally indicable) groups, and $r$ is a sixth or higher power, then $Q$ will have a decidable word problem; see 
\cite{Brodskii84, Howie81, HowieFourthPowers, ShortThesis}.

As we explain in the introduction to \cite{GrayLevineArxiv251207690}, a combination of \cite[Corollary 4.29]{Garreta2020} together with the recent breakthrough results \cite{Levent2026, koymans2024hilbert}, proves that  that if $G$ is a finitely generated virtually nilpotent group that is not virtually abelian then $G$ has an undecidable Diophantine problem. Since every virtually Heisenberg group is a finitely generated virtually nilpotent group that is not virtually abelian it follows that every virtually Heisenberg group has an undecidable Diophantine problem.  

In this section we shall describe a one-relator product of \(\Z\) and \(C_2\) with undecidable Diophantine problem. We will prove that the group in question is virtually the Heisenberg group, and then by the comments in the previous paragraph it will follow that the group has an undecidable Diophantine problem. However, once we have shown that it is indeed virtually the Heisenberg group, the remainder of the direct proof is very short, and shows that it is reducible to Hilbert's tenth problem over the ring of integers, so we include the proof. 

We start by recalling that the Heisenberg group is cyclically presented.

\begin{lem}
  \label{lem:Heisenberg-pres}
  The group defined by the presentation \(\langle x, y \mid [x, yxy^{-1}],
  [y, xyx^{-1}]\rangle\) is isomorphic to the Heisenberg group.
\end{lem}

\begin{proof}
  Recall that a standard presentation for the Heisenberg group is
  \(\langle a, b \mid [a, [a, b]], [b, [a, b]] \rangle\).
  We have \([a, [a, b]] = a aba^{-1} b^{-1} a^{-1} bab^{-1}a^{-1} = a [a,
  ba^{-1} b^{-1}] a^{-1}\), which is equivalent to the relation \([a, ba^{-1}
  b^{-1}]\). This states that \(a\) commutes with \(ba^{-1}b^{-1}\), which
  is equivalent to the relation \(a\) commutes with \((ba^{-1}b^{-1})^{-1} =
  bab^{-1}\); that is \([a, bab^{-1}]\).
  By symmetry, the relation \([b, [a, b]]\) is equivalent to \([b,
  aba^{-1}]\). Thus the Heisenberg group is isomorphic to \(\langle a, b
  \mid [a, bab^{-1}], [b, aba^{-1}] \rangle\), which has the desired form.
\end{proof}

We next define our one-relator product, and show that it is indeed virtually
	the Heisenberg group. 

\begin{lem}
  \label{lem:Heisenberg-subgp}
  The group
  \[
    K_2 = \langle x, t \mid [x, txt^{-1} x tx^{-1} t^{-1}], t^2 \rangle
  \]
  has a subgroup \(\langle x, txt^{-1} \rangle\) which is isomorphic to
  the Heisenberg group and transversal \(\{1, t\}\).
\end{lem}

\begin{proof}
Let $H = \langle x, y \mid [x, yxy^{-1}], [y, xyx^{-1}]\rangle$.  
Let $\phi$ be the automorphism of $H$ of order \(2\) that maps $x$ to $y$ and maps $y$ to $x$. Form the semidirect product of $H$ by $C_2$ with respect to this automorphism obtaining 
\[
H  \rtimes_{\phi} C_2 = 
\langle 
x, y, t \mid [x, yxy^{-1}]=1, \; [y, xyx^{-1}]=1, \; t^2=1, \; txt^{-1} = y, \; tyt^{-1}=x
\rangle.
\]         
We claim that $K_2 \cong H  \rtimes_{\phi} C_2 $. This can be proved by performing Tietze transformations to the presentation of the latter, first eliminating $y$ to obtain    
\[
H  \rtimes_{\phi} C_2 \cong 
\langle 
x, t \mid [x, (txt^{-1} )x(txt^{-1} )^{-1}]=1, \; [txt^{-1} , x(txt^{-1} )x^{-1}]=1, \; t^2=1, \; t(txt^{-1} )t^{-1}=x
\rangle
\] 
then observing that the second and fourth relations in this presentation are consequences of the other two and thus   
\[
H  \rtimes_{\phi} C_2 \cong  
\langle 
x, t \mid [x, (txt^{-1} )x(txt^{-1} )^{-1}]=1, \; t^2=1
\rangle = K_2.   
\] 
This shows that $H  \rtimes_{\phi} C_2 \cong K_2$ via the isomorphism $\theta$ that maps  $x \mapsto x$, $t \mapsto t$ and $y \mapsto txt^{-1}$. Since $H$ naturally embeds in the semidirect product $H  \rtimes_{\phi} C_2$ it follows via the isomorphism $\theta$ that the subgroup of $K_2$ generated by $\langle x, txt^{-1} \rangle$ is isomorphic to the Heisenberg group $H$. 
Also since $H$ is an index two subgroup of $H  \rtimes_{\phi} C_2 \cong K_2$ with transversal $\{1, t \}$ it follows via the isomorphism $\theta$ that $\langle x, txt^{-1} \rangle$ is an index two subgroup of $K_2$ with transversal $\{1, t\}$.        
This completes the proof of the lemma. 
\end{proof}

At this point, as mentioned earlier, we could now appeal to \cite[Corollary 4.29]{Garreta2020} and \cite{Levent2026, koymans2024hilbert}, to complete the proof. We instead include a self-contained short proof of the result below that reduces the result directly to Hilbert's tenth problem over the ring of integers.

\begin{theorem}
There is a one-relator product of the form $(\Z \ast C_2) / \langle \langle r \rangle \rangle$ with undecidable  
Diophantine problem. Specifically the group
  \[
    K_2 = \langle x, t \mid [x, txt^{-1} x tx^{-1} t^{-1}], t^2 \rangle
  \]
  has an undecidable Diophantine problem.
\end{theorem}

\begin{proof}
  We proceed by first showing \(K_2\) contains a copy of the Heisenberg group as
  a finite index subgroup, and then showing that this subgroup is equationally
  definable. Since the Diophantine problem in the Heisenberg group is
  undecidable \cite[Theorem 4]{duchin_liang_shapiro}, the result will follow.

  Let \(y = txt^{-1}\). Note that \(tyt^{-1} = x\). \cref{lem:Heisenberg-subgp}
  tells us that \(\langle x, y \rangle\) is isomorphic to the Heisenberg group.
  We show that \(C_{K_2}(x yx^{-1}y^{-1}) = \langle x, y \rangle\). Naturally,
  \(\langle x, y \rangle \leq C_{K_2}(xyx^{-1}y^{-1})\). So let \(g t^\delta
  \in C_{K_2}(xyx^{-1}y^{-1})\), where \(g \in \langle x, y \rangle\) and
  \(\delta \in \{0, 1\}\). Note that every element of \(K_2\) can be written in
  this form. If \(\delta = 0\) then \(gt^\delta \in \langle x, y \rangle\), so
  assume \(\delta = 1\).  So
  \[
    xyx^{-1} y^{-1} \cdot gt = gt \cdot xyx^{-1}y^{-1}
    = g yxy^{-1}x^{-1}t = g (xyx^{-1}y^{-1})^{-1} t
  \]
  Thus \(xyx^{-1} y^{-1} g = g (xyx^{-1}y^{-1})^{-1}\). Noting that 
  \((xyx^{-1}y^{-1})^{-1}\) is central in \(\langle x, y \rangle\), we have 
  that \(xyx^{-1} y^{-1} =  (xyx^{-1}y^{-1})^{-1}\), a contradiction as
  the subgroup generated by the commutator of the generator of the Heisenberg
  group is isomorphic to \(\mathbb{Z}\), and thus non-trivial and torsion-free.
  Thus \(C_{K_2}(x yx^{-1}y^{-1}) = \langle x, y \rangle\), and so
  \(\langle x, y \rangle\) is definable in \(K_2\), as required.
\end{proof}

Naturally, our single example of a one-relator product leads to various further
questions. We only used the cyclic groups \(\Z\) and \(C_2\); one can ask

\begin{qu}
	If \(C\) and \(C'\) are non-trivial cyclic groups, is there a one-relator
	product of \(C\) and \(C'\) with undecidable Diophantine problem?
\end{qu}

Additionally, by \cite[Theorem 1.1]{Levine2022} the decision question asking if a
single equation admits a solution is decidable in \(K_2\). This leads us to ask
if an example exists where this is not the case. A positive answer to this question
would of course show that the conjugacy problem in all one-relator groups is decidable.

\begin{qu}
	Is it decidable in a given one-relator product of two cyclic groups whether a given
	single equation admits a solution?
\end{qu}

\section{Diophantine problem with single subgroup constraints in one-relator groups}
\label{sec:subgroup-constraints}

In this section we move away from the `plain' Diophantine problem to the Diophantine problem with constraints. Adding constraints to Diophantine problems is not new; Diophantine problems in virtually free groups with rational constraints are known to be decidable \cite{dahmani_guirardel}, as are hyperbolic groups with quasi-isometrically emeddable rational constraints \cite{dahmani_guirardel}. Right-angled Artin groups (and the monoid analogue of trace monoids) do work with `normalised regular' constraints \cite{DiekertMuscholl}, but not with arbitrary rational constraints. The proof of this fact, which is identical to the trace monoids analogue given by Muscholl \cite{MuschollThesis}, only requires subgroup constraints in the right-angled Artin group \(F_2 \times \Z\), but three such subgroups are required. We shall give an alternative proof of that result below in Lemma~\ref{lem:AP3DPG} where we prove undecidability by showing a reduction to yet another type of constraint: abelianisation constraints. First shown to be undecidable for free monoids by B\"{u}chi and Senger \cite{BuchiSenger}, these have recently been studied in various classes of groups including right-angled Artin groups and hyperbolic groups \cite{CiobanuGarreta}, and virtually abelian groups \cite{CiobanuEvettsLevine}. Arguments similar to those used to prove Diophantine problems with abelianisation constraints are undecidable have had some success in the `plain' Diophantine problem, such as for Thompson's group \(F\) \cite{ElliottLevine} and in \cref{GBS-thm}.

\subsection*{Single subgroup constraints}
The weakest constraint one might consider adding to the Diophantine problem is one with constraints that variables may lie in one single fixed finitely generated subgroup. 

This property mixes the subgroup membership problem together with the Diophantine problem. More precisely, given a finitely generated group $G$ and a finitely generated subgroup $H$ we say \emph{$G$ has decidable Diophantine problem with $H$-constraints} if there as an algorithm that takes any finite system of equations over $G$ together with a finite list of constraints of the form $X_1, \ldots, X_k \in H$ for variables $X_i$ and decides whether or not it has a solution. This notion gives an approach to the plain Diophantine problem. Indeed, if $H$ is equationally definable in $G$ and $G$ has and undecidable Diophantine problem with $H$-constraints, then $G$ will have an  undecidable Diophantine problem. Note that it is immediate from the definition that if a group $G$ has a decidable Diophantine problem with $H$-constraints then the membership problem for $H$ in $G$ must be decidable.   

The main result of this section is the following. 

\begin{theorem}\label{thm:singleSubgroupConstraintsOneRelatorGroup} Let
\[
	G = \langle c, q \mid c(qcq^{-1}) = (qcq^{-1})c \rangle.
\] 
There is a single fixed finitely presented subgroup $H \leq G$ such that 
$H$ is isomorphic to a free group of rank $6$
and the Diophantine problem for $G$ with $H$-constraints is undecidable.    
  \end{theorem}
Before proving this result we make some remarks about it. 

The group $G$ does have decidable subgroup membership problem. Indeed, $G$ is a 3-manifold group by \cite{Burns1987} and every 3-manifold group has a decidable subgroup membership problem; see \cite{FriedlWilton}. In particular, membership in the subgroup $H$ of $G$ in the statement of the result is decidable.        

It is currently open whether all one-relator groups have decidable subgroup membership problem; see \cite[Problem 1.8.8]{CFMarcoOneRelatorSurvey}. 

The group in the result above fits into several other families which gives the following corollary. 

\begin{cor}\label{cor:OtherClasses} The group $ G = \langle c, q \mid c(qcq^{-1}) = (qcq^{-1})c \rangle $ 
is a one-relator, free-by-cyclic, 3-manifold group that contains a fixed finitely generated subgroup $H \leq G$ such that the Diophantine problem for $G$ with $H$-constraints is undecidable. \end{cor}

We currently do not know whether or not $G$ has a decidable Diophantine problem. If the subgroup $H$ in the above theorem were equationally definable in $G$ then it would show that $G$ has undecidable Diophantine problem, but it is not clear whether or not it is. As an aside, related to the question of whether $G$ has decidable Diophantine problem, note that another presentation for the group $G$ is    
\[
  \langle c, q \mid (q^icq^{-i})(q^jcq^{-j}) = (q^jcq^{-j}) (q^icq^{-i}) \; (|i-j| \leq 1) \rangle
\]
If we take this presentation and remove the restriction on $i$ and $j$ we obtain the infinitely presented wreath product     
\[
  \langle c, q \mid (q^icq^{-i})(q^jcq^{-j}) = (q^jcq^{-j}) (q^icq^{-i}) \; i,j \in \Z \rangle \cong \Z \wr \Z.
\]
It was recently proved by Dong \cite{Dong25} that this group $\Z \wr \Z$ has an undecidable Diophantine problem. This leads naturally to the question 
of whether there exists an integer $m$ such that the group      
\[
  \langle c, q \mid (q^icq^{-i})(q^jcq^{-j}) = (q^jcq^{-j}) (q^icq^{-i}) \; (|i-j| \leq m) \rangle
\]
has an undecidable Diophantine problem. 

For the proof of \cref{thm:singleSubgroupConstraintsOneRelatorGroup}, and for other results in this section, we shall need the general notion of Diophantine problems with constraints which we now formally define. 

\begin{dfn}
		Let \(\mathscr{C}\) be a collection of subsets of \(G\). A \emph{system
		of equations with \(\mathscr{C}\)-constraints} \(\mathcal E\) in \(G\)
		is a system \(\mathcal F\) of equations in \(G\) together with finitely
		many statements of the form \(X \in C\) where \(X\) is a variable in
		\(\mathcal F\) and \(C \in \mathscr{C}\). A \emph{solution} to
		\(\mathcal E\) is a homomorphism \(\phi\) which is a solution to
		\(\mathcal F\) and such that for each statement of the form \(X \in C\),
		we have \(\phi(X) \in C\) is satisfied.

		Particular cases of interest are when:
		\begin{enumerate}
			\item \(\mathscr{C} = \{[G, G]\}\) - the Diophantine problem with
					\emph{abelianisation constraints};
			\item \(\mathscr{C}\) is the set of all finitely generated
					subgroups (resp. rational subsets) - the Diophantine
					problem with \emph{subgroup (resp. rational) constraints};
			\item \(\mathscr{C} = \{H\}\), for some fixed finitely generated subgroup
					\(H \leq G\) - the Diophantine problem with \emph{\(H\)-constraints}.
		\end{enumerate}
\end{dfn}

Before we begin computing definable sets in \(G\), we first note two different
versions of it.

  \begin{lem} \label{lem:H-K-iso}
    There is an isomorphism \(\Theta\) from \(K = \langle a, b, s \mid sas^{-1} = ab, sbs^{-1} = b\rangle\) to \(G = \langle c, q \mid c(qcq^{-1}) = (qcq^{-1}) c \rangle\) defined by
    \[
        \Theta(a) = cq^{-1}, \quad \Theta(b) = qcq^{-1} c^{-1}, \quad \Theta(s) = c.
    \]
    In addition, let \(c_i = q^i c q^{-i}\), where \(i \in \mathbb{Z}\). Then
    \(H = \langle \{c_i \mid i \in \mathbb{Z}\} \rangle\) is the
    right-angled Artin group on an infinite line, and \(G \cong H \rtimes_\phi \Z\), where
    \(\phi \in \Aut(H)\) is defined by \(c_i \mapsto c_{i + 1}\).
  \end{lem}

\begin{proof} The fact that $K \cong G$ is proved in \cite[Section~2]{Burns1987}. It is straightforward to prove $\Theta$ is an isomorphism by eliminating $b$ from the presentation for $K$, then adding the redundant generators $c,q$ and relations $c=s$ and $q= a^{-1}c$. It is also observed in \cite[Section~2]{Burns1987} that the subgroup $H$ has presentation $\langle c_i \mid c_i c_{i+1} = c_{i+1} c_i  \; (i \in \Z) \; \rangle$, that is, it is the right-angled Artin group on an infinite line. Finally, the group $H \rtimes_\phi \Z$ has presentation 
\[
\langle c_j, q \mid 
c_j c_{j+1} = c_{j+1} c_j, \; 
q c_j q^{-1} = c_{j+1} \; (j \in \Z) \; \rangle  
\]
which after eliminating all the redundant generators $c_i = q^i c q^{-i}$ where $i \neq 0$ gives the presentation \[
H \rtimes_\phi \Z \cong \langle c_0, q \mid c_0 (qc_0 q^{-1}) = (qc_0 q^{-1}) c_0 \rangle \cong  
\langle c, q \mid c (qc q^{-1}) = (qc q^{-1}) c \rangle = G. \qedhere
\]  
\end{proof}

As with many proofs of undecidability of Diophantine-like problems in groups, we
first compute a centraliser in \(G\).

\begin{lem}
		\label{Bob-gp-centraliser}
Let \(G = \langle c, q \mid c (qcq^{-1}) = (qcq^{-1})  \rangle\). Then
\(C_G(qcq^{-1}) = \langle c, qcq^{-1}, q^2 c q^{-2} \rangle\).
\end{lem}

\begin{proof}
		We begin by showing \(\langle c, qcq^{-1}, q^2 c q^{-2} \rangle \subseteq C_G(qcq^{-1})\).
		It is immediate from the presentation that \(c, qcq^{-1} \in C_G(qcq^{-1})\). We have
		that 
		\[
			qcq^{-1} \cdot q^2 c q^{-2} = qc qc q^{-1} q^{-1} = q^2cq^{-1} cq^{-1} = 
			q^2 c q^{-2} \cdot qcq^{-1}.
		\]
		Thus \(\langle c, qcq^{-1}, q^2 c q^{-2} \rangle \subseteq
		C_G(qcq^{-1})\). For the reverse containment, we note that if we define
		\(c_i = q^i c q^{-i}\) for all \(i \in \Z\), then the subgroup \(A\)
		generated by \(\{c_i \mid i \in \Z\}\) is the right-angled Artin group
		on the infinite line (that is, the Cayley graph of \(\Z\) and
		\(G\) is a semidirect product of \(A\) with \(\Z\), with \(\Z\)
		(generated by \(t\)) acting by \(q c_i q^{-1} = c_{i + 1}\), by
		\cref{lem:H-K-iso}.

		So let \(u q^k \in C_G(qcq^{-1}) = C_G(c_1)\), where \(u \in A\) and
		\(k \in \Z\). Then
			\(c_1 \cdot uq^k  = uq^k \cdot c_1 
						    = u c_{1 + k} q^k\).
		Since we are in a semidirect product, it follows that \(c_1 u =_A u
		c_{1 + k}\) (and redundantly \(k = k\)). Since the abelianisations of
		the words on either side of this equality must be equal, it follows
		that \(k = 0\). Thus \(c_1 u =_A u c_1\); that is, \(u \in C_A(c_1)\).
		By \cite[The Centralizer Theorem]{Servatius89} (see also
		\cite[Proposition 2.1]{DayWade19}), centralisers of generators of
		right-angled Artin groups are the parabolic subgroups generated by the
		neighbours of the generator in the underlying graph, together with the
		generator itself. Thus \(C_A(c_1) = \langle c_0, c_1, c_2 \rangle\).
		We can conclude that \(ut^k \in \langle c_0, c_1, c_2 \rangle\) and so
		\(C_G(c_1) \subseteq \langle c_0, c_1, c_2\rangle\). In other words,
		\(C_G(qcq^{-1}) \subseteq \langle c, qcq^{-1}, q^2 cq^{-2} \rangle\),
		as required.
\end{proof}

\begin{lem}\label{lem:retraction}
Let $L$ be a group and let $P \leq L$ be a retract via a homomorphism $\phi: L \rightarrow L$ such that $\phi(p) = p$ for all $p \in P$ and $\phi(L) = P$. Then for any subset $X$ of $L$, if $\phi(X) \subseteq X$ then $\langle X \rangle \cap P = \langle \phi(X) \rangle \leq P$. 
  \end{lem}
\begin{proof} 
Let $p \in P$ and let $x_1, \ldots, x_k \in X^{\pm 1}$.
If  
$p = x_1 \ldots x_k$ then $p = \phi(p) = \phi(x_1) \ldots \phi(x_k) \in \langle \phi(X) \rangle$. Conversely, if $p = \phi(x_1) \ldots \phi(x_k)$ then since $\phi(X) \subseteq X$ it follows that $p \in \langle X \rangle$ and thus $p \in \langle X \rangle \cap P$        
  \end{proof}

We also need the following straightforward lemmas about intersections of subgroups in $A(P_3) \cong F_2 \times \Z$. The following lemma is a straightforward consequence of the definitions.
  
\begin{lem} Exponent sums of elements of \(A(P_3) = \langle a, b, c \mid ab = ba, bc = cb \rangle\) are well-defined; that is, if \(u, v \in (\{a, b, c\}^\pm)^\ast\), are such that \(u = _{A(P_3)} v\), then \(\expsum_x(u) = \expsum_x(v)\) for each \(x \in \{a, b, c\}\). \end{lem}

\begin{lem}\label{Lem:DPSGAP3}
In the group \(A(P_3) = \langle a, b, c \mid ab = ba, bc = cb \rangle\) we have
\[
\langle ab, c \rangle \cap
    \langle a, bc \rangle \cap \langle a, c \rangle = [\langle a, c \rangle,
    \langle a, c \rangle].
\]
\end{lem}
\begin{proof}
    We begin by showing that \(\langle ab, c \rangle =
    \{u b^{\expsum_a(u)} \mid u \in \langle a, c \rangle\}\).

    To show the forward containment, let \(w \in \langle ab, c \rangle\). Pushing
    \(b\)s to the right in \(w\) gives a normal form \(w = ub^k\), where
    \(u \in \langle a, c \rangle\) and \(k \in \mathbb{Z}\). Then
    \[
        k = \expsum_b(ub^k) = \expsum_b(w) = \expsum_a(w) = \expsum_a(u),
    \]
    and so \(w = ub^{\expsum_a(u)}\). For the converse containment, let
    \(u \in \langle a, c \rangle\). Let \(\bar{u} \in \langle ab, c \rangle\)
    be the image of \(u\) under the map defined by \(a \mapsto ab\) and \(c \mapsto c\).
    Then \(\langle ab, c \rangle \ni \bar{u} =_{A(P_3)} ub^{\expsum_a(u)}\), as required.

    We have now showed that \(\langle ab, c \rangle =
    \{u b^{\expsum_a(u)} \mid u \in \langle a, c \rangle\}\); the fact that
    \(\langle a, cb \rangle =
    \{u b^{\expsum_c(u)} \mid u \in \langle a, c \rangle\}\) follows by symmetry. Note that
    \[
        \langle ab, c \rangle \cap \langle a, c \rangle = \{u b^{\expsum_a(u)} \mid u \in \langle a, c \rangle\} \cap \langle a, c \rangle
        = \{u \in \langle a, c \rangle \mid \expsum_a(u) = 0\}.
    \]
    By symmetry, \(\langle a, cb \rangle \cap \langle a, c \rangle = \{u \in \langle a, c
    \rangle \mid \expsum_c(u) = 0\}\). Thus
    \[
        \langle ab, c \rangle \cap \langle a, c \rangle \cap \langle a, cb \rangle
        = \{u \in \langle a, c
    \rangle \mid \expsum_a(u) = \expsum_c(u) = 0\} = [\langle a, c \rangle, \langle
    a, c \rangle]. \qedhere
    \] 
\end{proof}

We can now prove the main result of this section.

\begin{proof}[Proof of Theorem~\ref{thm:singleSubgroupConstraintsOneRelatorGroup}]
Let
\[
G = \langle c, q \mid c(qcq^{-1}) = (qcq^{-1})c \rangle.
\] 
Set 
$c_i = q^i c q^{-i}$. 
Then $G$ has the infinite presentation  
\[
G = \langle c_i \; (i \in \mathbb{Z}), q \mid 
qc_iq^{-1} = c_{i+1}, \;  
c_i c_{i+1} = c_{i+1} c_i \quad (i \in \mathbb{Z})
\rangle.
\]   
It follows that  
the subgroup of $G$ generated
by the set $\{c_i \mid i \in \Z\}$ is an infinite line right-angled Artin group with underlying
graph the standard Cayley graph of $\mathbb{Z}$. With this notation let 
\[
H = 
\langle 
c_0, c_2, 
c_4c_5, c_6, 
c_8, c_9c_{10}
\rangle. 
\] 
We claim that the Diophantine problem for $G$ with $H$-constraints 
is undecidable.  
Observe that:
\[
q^{-4}Hq^4 = 
\langle 
c_{-4}, c_{-2}, 
c_0c_1, c_2, 
c_4, c_5c_{6}
\rangle 
\] 
and 
\[
q^{-8}Hq^8 = 
\langle 
c_{-8}, c_{-6}, 
c_{-4}c_{-3}, c_{-2}, 
c_0, c_1c_{2}
\rangle. 
\] 
Now consider the intersection: \( H \cap q^{-4}Hq^4 \cap q^{-8}Hq^8\).

In any right-angled Artin group the intersection of two subgroups each
generated by some subset of the vertices is equal to the subgroup generated by
the intersection of that set of vertices; see
\cite[Lemma~3.1]{AntolinFoniqi22}. It follows that, working in the right-angled
Artin group on the infinite line $\{c_i \mid i \in \mathbb{Z} \}$ we have 
\[
\langle c_0, \ldots c_{10} \rangle \cap 
\langle c_{-4}, \ldots c_6 \rangle \cap 
\langle c_{-8}, \ldots c_2 \rangle 
= 
\langle c_0,c_1,c_2 \rangle. 
\]
Hence 
\[
H \cap q^{-4}Hq^4 \cap q^{-8}Hq^8 \leq \langle c_0,c_1,c_2 \rangle
\]
and thus 
\[
H \cap q^{-4}Hq^4 \cap q^{-8}Hq^8 = 
\langle c_0, c_2 \rangle \cap 
\langle c_0c_1, c_2 \rangle \cap 
\langle c_0, c_1c_2 \rangle.  
\]
Indeed, since  
\[
H \cap q^{-4}Hq^4 \cap q^{-8}Hq^8 \leq \langle c_0,c_1,c_2 \rangle
\]
it follows that 
\begin{align*} 
H \cap q^{-4}Hq^4 \cap q^{-8}Hq^8 & = 
H \cap q^{-4}Hq^4 \cap q^{-8}Hq^8 \cap \langle c_0,c_1,c_2 \rangle \\ 
& = 
(H \cap \langle c_0,c_1,c_2 \rangle) \cap 
(q^{-4}Hq^4  \cap \langle c_0,c_1,c_2 \rangle) \cap 
(q^{-8}Hq^8 \cap \langle c_0,c_1,c_2 \rangle). 
  \end{align*}
Now using the retraction from the subgroup $\langle \{ c_i: i \in \Z \} \rangle$ to the subgroup $\langle c_0, c_1, c_2 \rangle$ that fixes $c_0, c_1$ and $c_2$ and  sends every other $c_i \mapsto 1$, and applying Lemma~\ref{lem:retraction}, it follows that
\begin{align*} 
H  \cap \langle c_0,c_1,c_2 \rangle &= 
\langle 
c_0, c_2, 
c_4c_5, c_6, 
c_8, c_9c_{10}
\rangle
\cap \langle c_0,c_1,c_2 \rangle = \langle c_0, c_2 \rangle, \\
q^{-4}Hq^4  \cap \langle c_0,c_1,c_2 \rangle
& = 
\langle 
c_{-4}, c_{-2}, 
c_0c_1, c_2, 
c_4, c_5c_{6}
\rangle 
 \cap \langle c_0,c_1,c_2 \rangle = \langle c_0c_1, c_2 \rangle,  \\
q^{-8}Hq^8 \cap \langle c_0,c_1,c_2 \rangle
& = 
\langle 
c_{-8}, c_{-6}, 
c_{-4}c_{-3}, c_{-2}, 
c_0, c_1c_{2}
\rangle 
 \cap \langle c_0,c_1,c_2 \rangle = \langle c_0, c_1c_2 \rangle. 
\end{align*}

Hence 
\[
H \cap q^{-4}Hq^4 \cap q^{-8}Hq^8 = 
\langle c_0, c_2 \rangle \cap 
\langle c_0c_1, c_2 \rangle \cap 
\langle c_0, c_1c_2 \rangle,  
\]
as claimed.
But now, applying Lemma~\ref{Lem:DPSGAP3} we obtain 
\[
H \cap q^{-4}Hq^4 \cap q^{-8}Hq^8 = 
\langle c_0, c_2 \rangle \cap 
\langle c_0c_1, c_2 \rangle \cap 
\langle c_0, c_1c_2 \rangle = 
[\langle c_0, c_2 \rangle, \langle c_0, c_2 \rangle]. 
\]

From \cref{Bob-gp-centraliser}, we have \(\langle c_0, c_1, c_2 \rangle = C_G(c_1)\), and so \(\langle c_0, c_1, c_2 \rangle\) is equationally definable in \(G\), along with constraints of the form \(X \in [\langle c_0, c_2 \rangle,
\langle c_0, c_2 \rangle]\). Additionally, \([\langle c_0, c_1, c_2 \rangle, \langle c_0, c_1, c_2 \rangle] = [\langle c_0, c_2 \rangle, \langle c_0, c_2 \rangle]\) and so we can define \(\langle c_0, c_1, c_2 \rangle\) into \(G\), along with abelianisation constraints in \(\langle c_0, c_1, c_2 \rangle\). Thus if the Diophantine problem with \(H\)-constraints in \(G\) is decidable, then the Diophantine problem with abelianisation constraints is decidable in \(\langle c_0, c_1, c_2 \rangle \cong F_2 \times \Z\), which directly contradicts \cite[Theorem A]{CiobanuGarreta}. 

To complete the proof we need to prove that the group   
$H$
is isomorphic to a free group of rank $6$.    
For this first note that the subgroup 
$K = \langle c_j \mid 0 \leq j \leq 10 \rangle$
is isomorphic to the RAAG with defining graph a line $\Gamma$ with 11 vertices
and 10 edges. Now there is an automorphism $\Psi$ of $K$ mapping $c_4 \mapsto
c_4c_5$, $c_{10} \mapsto c_9c_{10}$ and fixing every other vertex since this is
an example of a $\Gamma$-legal transvection in the sense of e.g.
\cite[Section~2.5]{Vogtmann15}. Since $\Psi$ is an automorphism of the RAAG
$A(\Gamma)$ it follows that the subgroup of $A(\Gamma)$ generated by $\{ c_0,
c_2, c_4, c_6, c_8, c_{10} \}$ is isomorphic to the subgroup generated by  $\{
c_0, c_2, c_4c_5, c_6, c_8, c_9c_{10} \}$. But since it is a standard parabolic
subgroup, the subgroup of $A(\Gamma)$ generated by $\{ c_0, c_2, c_4, c_6, c_8,
c_{10} \}$ is the free group of rank 6 by e.g. \cite[Lemma 3.1]{AntolinFoniqi22}. This completes the proof that $H$ is
isomorphic to a free group of rank $6$.     
\end{proof}

\begin{rmk} It is natural to ask whether \cref{thm:singleSubgroupConstraintsOneRelatorGroup} might extend to all Hydra groups \cite{HydraGroups} -- showing that the Diophantine problem with single subgroup constraints in any Hydra group is undecidable. For example one can ask whether the result holds for the following group: 
\[
L = \langle 
x,a,b,s \mid sxs^{-1} = xa, sas^{-1} = ab, sbs^{-1}=b 
\rangle.
\]
Now a key step in the proof of \cref{thm:singleSubgroupConstraintsOneRelatorGroup} for the group 
\[
\langle 
a,b,s \mid sas^{-1} = ab, sbs^{-1}=b 
\rangle.
\]
is \cref{Bob-gp-centraliser} which in that notation shows (up to conjugation by $q$) that $C(c) = \langle q^{-1} c q, c, qcq^{-1} \rangle$ which in the notation of the presentation above shows $C(s) = \langle s, b, ab^{-1}a^{-1} \rangle$. This relates to the fact that the fixed point subgroup of the automorphism $a \mapsto ab$, $b \mapsto b$ of the free group on $\{a,b\}$ is rank two and generated by $\{b, ab^{-1}a^{-1} \}$. The same proof should work for the group $L$ if it is also true in $L$ that $C(s) = \langle s, b, ab^{-1}a^{-1} \rangle$ which seems plausible. Related to this is the question of what is the rank of the fixed point subgroup of the automorphism of $F(x,a,b)$ mapping $x \mapsto xa$, $a \mapsto ab$ and $b \mapsto b$.      
  \end{rmk}

\subsection*{Multiple subgroup constraints: negative and positive results}
We end this section by making some remarks about the Diophantine problem with multiple subgroup constraints in one-relator groups. It is immediate from the definition that if this problem is decidable for a group $G$ and $H \leq G$ then it is also decidable for $H$ (since we can add membership in $H$ as an additional constraint).   

\subsubsection*{Undecidability results} A helpful lemma for analysing this is a result that goes back to work of Muscholl \cite{MuschollThesis}. To keep our paper self-contained, and also since the argument was needed in the proof of Theorem~\ref{thm:singleSubgroupConstraintsOneRelatorGroup}, we give short proof of Muscholl's result here, using an alternative proof that goes via abelianisation constraints.  

\begin{lem}\label{lem:AP3DPG} In the group \(A(P_3) = \langle a, b, c \mid ab = ba, bc = cb \rangle\) the Diophantine problem with subgroup constraints in the three subgroups
\[
  \langle ab, c \rangle,
  \langle a, bc \rangle,
  \langle a, c \rangle 
\]
is undecidable. Hence the Diophantine problem with subgroup constraints is undecidable in any group $G$ that embeds the group $A(P_3) = F_2 \times \Z$.  
  \end{lem}
\begin{proof} Suppose seeking a contradiction that the Diophantine problem with subgroup constraints in the three subgroups
\[
  \langle ab, c \rangle,
  \langle a, bc \rangle,
  \langle a, c \rangle 
\]
is decidable in $A(P_3)$. It follows from Lemma~\ref{Lem:DPSGAP3} that 
\[
\langle ab, c \rangle \cap
    \langle a, bc \rangle \cap \langle a, c \rangle = [\langle a, c \rangle,
    \langle a, c \rangle].
\]
Additionally, since $b$ is central in the group we have  \([\langle a,b,c \rangle, \langle a,b,c \rangle] = [\langle a, c \rangle, \langle a, c \rangle]\). Hence the Diophantine problem with abelianisation constraints is decidable in $A(P_3)$ but this contradicts \cite[Theorem A]{CiobanuGarreta}. \end{proof}

This lemma has some consequences about the Diophantine problem with subgroup constraints that are worth recording. It is known that every one-relator group with centre is virtually $F_n \times \Z$; 
see \cite{Pietrowski1974} and  \cite[Proposition 4.1]{levitt2015generalized}.
Hence this immediately gives:    

\begin{cor} A one-relator group with centre has an undecidable Diophantine problem with subgroup constraints, unless it is virtually abelian. \end{cor}

Lemma~\ref{lem:AP3DPG} can be used to classify the Artin groups with decidable Diophantine problem with subgroup constraints.   

\begin{theorem}\label{thm:ArtinGroup}
Let \(G\) be a finitely generated Artin group. Then the following are equivalent:
    \begin{enumerate}
      \item \(G\) is a free product of finitely many free abelian groups;
      \item \(G\) has decidable existential theory with rational constraints;
      \item \(G\) has decidable Diophantine problem with subgroup constraints.
  \end{enumerate} 
  \end{theorem}
\begin{proof} 
	\((2) \implies (3)\): Immediate.

	\((3) \implies (1)\): By \cref{lem:AP3DPG}, any group that contains \(A(P_3)\)
	has an undecidable Diophantine problem with subgroup constraints. So it
	suffices to show that \(A(P_3)\) embeds into any Artin group that is not a
	free product of finitely many free abelian groups. We first consider the
	case of dihedral (two-generated) Artin groups. 

	Consider a (non-RAAG) dihedral Artin group \(G = \langle a, b \mid
	\underbrace{aba \cdots}_{m} = \underbrace{bab \cdots}_m \rangle\), where
	\(m \geq 2\). It is well-known that \(G\) is admits \(F_n \times
	\mathbb{Z}\) as a finite index subgroup, for some \(n \geq 2\) (the exact
	value of which depends on \(m\). For a brief proof of this fact, see for
	example \cite[Section 2]{CiobanuHoltRees}. We have thus shown that all
	non-RAAG dihedral Artin groups contain a copy of \(A(P_3)\), and so any
	non-RAAG Artin group contains a copy of \(A(P_3)\). Thus the only Artin
	groups with a decidable Diophantine problem with subgroup constraints are
	RAAGs. Naturally, any RAAG whose underlying graph is not a disjoint union
	of complete graphs will contain \(A(P_3)\) as a parabolic subgroup, and
	thus cannot have a decidable Diophantine problem with subgroup constraints.
	Thus an Artin group with a decidable Diophantine problem with subgroup
	constraints has an underlying graph that is the disjoint union of
	complete graphs; that is, it is a free product of free abelian groups.
\end{proof}

\begin{rmk} 
The analogous result to Theorem~\ref{thm:ArtinGroup} but for Artin monoids also holds, and can be proved using similar methods.  
  \end{rmk}

Lemma~\ref{lem:AP3DPG} also helps make some initial steps towards classifying
free-by-cyclic groups with decidable Diophantine problem with subgroup
constraints as we now explain. By a free-by-cyclic group we mean one of the
form $F_n \rtimes \Z$ for some $n \geq 2$. It is currently an open problem
whether all free-by-cyclic groups have decidable Diophantine problem. On the
other hand, it is know that every hyperbolic free-by-cyclic group does have a
decidable Diophantine problem, since all hyperbolic groups do.
Outer automorphisms of a finitely generated, rank $n$ free group $F_n$ come in two kinds according to the growth rate of conjugacy classes of $F_n$. The names for these two types is exponential and polynomial. The following result relates the two cases:    

\begin{theorem} Suppose $\Phi \colon F_n \rightarrow F_n$ grows exponentially. Then $G_\Phi$ is hyperbolic relative to a finite collection of subgroups, each of the form $F \rtimes_{\psi} \Z$ for a finite rank free group $F$ and a polynomially growing automorphism $\psi: F \rightarrow F$. \end{theorem}

The theorem above follows from the papers: \cite{gautero2007mapping},
\cite{Ghosh23} and \cite{DahmaniRuoyu}. Combined with results about the
Diophantine problem for relatively hyperbolic groups \cite{Dahmani}, 
these results suggest that a key step for proving that all free-by-cyclic groups have decidable
Diophantine problems would be to first resolve it in the polynomial case. We can
apply Lemma~\ref{lem:AP3DPG} to show that if we add subgroup constraints the
problem is always undecidable in the polynomial case.  

\begin{theorem}\label{thm:FBCPolynomial}
Let $G=F_n \rtimes_\varphi \mathbb{Z}$ for some $n \geq 2$. If \(\varphi\) is polynomially growing, then \(G\) has undecidable Diophantine problem with subgroup constraints.
\end{theorem}
\begin{proof} It follows from \cite[Theorem~3.4]{kudlinska2024subgroup} 
that any such group is either virtually a direct product $F_n \times \Z$ or else it contains $A(P_4)$. In both cases $G$ embeds $A(P_3)$ and now the result follow from Lemma~\ref{lem:AP3DPG}. \end{proof}
 
\subsubsection*{Decidability results} It has recently been proved 
\cite{linton2025geometry} that a finitely generated one-relator group $\langle A \mid w=1 \rangle$ is locally quasi-convex hyperbolic if and only if $\pi(w) \neq 2$, where $\pi(w)$ denotes the primitivity rank of $w$ (see that paper for the definition of primitivity rank). 
One corollary of that result is that any one-relator group with torsion is locally quasi-convex hyperbolic. We note that Puder provided an algorithm to compute the primitivity rank \cite{puder2014primitive}. 
Combining Linton's result with other results in the literature \cite{dahmani_guirardel} gives: 

\begin{proposition} Let $G = \langle A \mid w=1 \rangle$ be a finitely generated one-relator group with primitivity rank $\pi(w) \neq 2$. Then $G$ the Diophantine problem with subgroup constraints is decidable in $G$. 
In particular every one-relator group with torsion and every 
hyperbolic surface group has decidable Diophantine problem with subgroup constraints.
\end{proposition}
The last sentence in the previous proposition follows from the fact that one-relator groups with torsion and hyperbolic surface groups are one-relator where the relator word has primitivity rank different from $2$; see \cite{linton2025geometry, puder2014primitive}.  

Related to this, note that the subgroup membership problem is open in general for one-relator groups. It is open whether the Diophantine problem is decidable in $BS(1,n)$. Gersten's conjecture (see \cite[Conjecture 2.6.1]{CFMarcoOneRelatorSurvey}) states that if a one-relator group contains no $BS(1,n)$ for any $n \neq 0$ then $G$ is hyperbolic. This prompts the question:  
   
\begin{qu} Does $BS(1,n)$, with $n > 1$, have a decidable Diophantine problem with subgroup constraints? \end{qu}

In fact $BS(1,n)$ is known to have a decidable rational subset membership problem \cite{Cadilhac2020} so one can ask more generally whether is has a decidable Diophantine problem with rational constraints.

\section{An $F_3$ by $F_2$ group }\label{sec:FreeByFree} 
As discussed in the previous section, it is an open question whether all free-by-cyclic groups 
$F_n \rtimes \Z$ have decidable Diophantine problems.  While that question remains open, in this section we shall give a relatively simple example of a group of the form $F_3 \rtimes F_2$ with an undecidable Diophantine problem.  

As mentioned in the introduction above, as far as the authors are aware, the conjugacy problem is open for groups of the form $F_3 \rtimes F_2$. In the example constructed in \cite{BogopolskiMartinoVentura} of a group of the form $F_3 \rtimes F_{14}$ with undecidable conjugacy problem, the number $14$ arises from an input to the construction which is a $2$-generator $12$-relator group with undecidable word problem due to Borisov \cite{Borisov1969}.   This leads naturally to the question: 

\begin{qu} What is the smallest value of $m$ such that there is a group of the form $F_3 \rtimes F_m$ with undecidable conjugacy problem? \end{qu}

We know that $m=1$ since free-by-cyclic groups have decidable conjugacy problem. On the other hand as mentioned above $m=14$ is possible. 
Related to this, building on Miller's construction \cite{MillerBook}, Wei{\ss} proves in \cite[Theorem 5.8]{WeissThesis} that there are finitely generated subgroups $A, B \leq F_2$ and isomorphism $\phi:A \rightarrow B$ such that the HNN-extension $\langle F_2, t \mid tat^{-1} = \phi(a) \rangle$ has undecidable conjugacy problem.  

Returning to the main result in this section, our aim is to prove Theorem~\ref{thm:FreeByFree} which states that the group  
\[
\langle a, b,c, s, t \mid sas^{-1} = ab, sbs^{-1} = b, scs^{-1} = c, tat^{-1} = ba, tbt^{-1} = b, tct^{-1} = c \rangle
\]
has an undecidable Diophantine problem.

  The study of this free-by-free group begins by computing various
  centralisers. To do this, we first work in the similar free-by-cyclic
  group we studied in Section~\ref{sec:subgroup-constraints}. To avoid
  confusion between this related free-by-cyclic group, and the `principal'
  group of this section \(G\), we will refer to the former as \(L\)
  throughout this section.

\begin{lem} 
    \label{Bob-gp-cent}
	Let \(K = \langle a, b, s \mid sas^{-1} = ab, sbs^{-1} = b \rangle\), and
	\(L = \langle c, q \mid c(qcq^{-1}) = (qcq^{-1})c \rangle\) be the isomorphic copy
	of \(K\) from \cref{lem:H-K-iso}. Then
    \begin{enumerate}
       \item \(C_L(c_i) = \langle c_{i - 1}, c_i, c_{i + 1} \rangle\);
       \label{item:H-cent-ci}
       \item \(C_K(s) = \langle aba^{-1}, b, s \rangle\).
       \label{item:L-cent-r}
    \end{enumerate}
  \end{lem}

\begin{proof}
    \eqref{item:H-cent-ci}: First note that from
	the right-angled Artin group presentation of \(H\), we have that \(\langle
	c_{i-1}, c_{i}, c_{i + 1} \rangle \subseteq C_L(c_i)\). Now let \(u q^k \in
	C_L(c_i)\), where \(u \in H\) and \(k \in \mathbb{Z}\). Then \(c_i
	uq^k =_L uq^k c_i =_L uc_{k + i} q^k\), and so \(c_i u =_{H} u c_{k
	+ i}\); that is \(c_i\) is conjugate in \(H\) to \(c_{k + i}\). It
	follows that \(c_i\) and \(c_{k + i}\) are conjugate in the abelianisation
	of \(H\); that is, they are equal in the abelianisation of \(H\). But
	generators in \(H\) are mapped to generators in the infinite rank
	free group, and so the images of \(c_i\) and \(c_{i + k}\) cannot be equal
	in the abelianisation unless \(c_i = c_{i + k}\). Thus \(k = 0\), and so
	\(C_H(c_i) \subseteq H\). Since the \(C_{H}(c_i) = \langle
	c_{i - 1}, c_i, c_{i + 1} \rangle\) (\cite[The Centralizer Theorem]{Servatius89}, see also \cite[Proposition 2.1]{DayWade19}), the result follows.

        \eqref{item:L-cent-r}: We compute \(C_K(s)\) by considering its image in
    the isomorphic group \(L\), using the isomorphism from \cref{lem:H-K-iso}.
    So we instead need to compute \(C_L(c)\). From \eqref{item:H-cent-ci},
    \(C_L(c) = C_L(c_0) = \langle c_{-1}, c_0, c_1 \rangle = \langle
    q^{-1}cq, c, qcq^{-1} \rangle\). The preimage of this subgroup under the
    isomorphism from \cref{lem:H-K-iso} is \(\langle s^{-1} a s a^{-1} s,
    s, bs \rangle = \langle a b^{-1} a^{-1}s, s, bs \rangle =
    \langle aba^{-1}, b, s \rangle\).
\end{proof}

	We now study the object of this section, and as usual, begin by computing
	centralisers. 

	\begin{lem}
		\label{F2byF2-cents}
		Let \(G = \langle a, b,c, s, t \mid sas^{-1} = ab, sbs^{-1} = b, scs^{-1} = c, tat^{-1} = ba,
		tbt^{-1} = b, tct^{-1} = c \rangle\). Then:
		\begin{enumerate}
				\item \(C_G(s) = \langle aba^{-1}, b, c, s \rangle\);
				\label{item:F2byF2-cent-s}
				\item \(C_G(t) = \langle a^{-1} ba, b, c, t \rangle\);
				\label{item:F2byF2-cent-t}
				\item \(C_G(b) = \langle b, s, t \rangle\);
						\label{item:F2byF2-cent-b}
				\item \(C_G(c) = \langle c, s, t \rangle\);
						\label{item:F2byF2-cent-c}
				\item \(C_G(st^k) = \langle b, c, st^k \rangle\) for all \(k \in \Z \setminus \{0\}\).
						\label{item:F2byF2-cent-stk}
		\end{enumerate}
	\end{lem}
	
	\begin{proof}
		We begin by showing \eqref{item:F2byF2-cent-s};
		\eqref{item:F2byF2-cent-t} is analogous. First note that \(b,c, s \in
		C_G(s)\) from the presentation, and \(saba^{-1} = abbb^{-1} a^{-1} s =
		aba^{-1}s\), and so \(aba^{-1} \in C_G(s)\). We have thus shown that
		\(\langle aba^{-1}, b, c, s \rangle \subseteq C_G(s)\). For the reverse
		containment, noting that \(G\) is a semidirect product of \(\langle a,
		b,c  \rangle\) with \(\langle s, t \rangle\), every element of \(G\) is
		expressible in the form \(uv\) with \(u \in \langle a, b, c \rangle\)
		and \(v \in \langle s, t \rangle\).  So let \(uv \in C_G(s)\). Then
		\(uvs = suv = \phi_s(u) \cdot sv\), where \(\phi_s \in \Aut(\langle a,
		b, c \rangle)\) is the automorphism defined by conjugation by \(s\).
		Then \(sv = vs\) and so \(v \in \langle s \rangle\), since centralisers
		in free groups are cyclic. Thus \(v = s^i\) for some \(i \in \Z\).
		Additionally, \(u = \phi_s(u)\), and so \(u \in \Fix(\phi_s)\). We have
		thus shown that \(C_G(s) = \langle \Fix(\phi_s), s \rangle\). To
		complete the proof of (1), it thus suffices to show that \(\Fix(\phi_s)
		= \langle aba^{-1}, b, c \rangle\). Since \(\phi_s\) fixes \(\langle a,
		b\rangle\) and \(\langle c \rangle\) setwise, it suffices (see for
		example \cite[Theorem 3.1]{MartinoVentura}) to compute
		\(\Fix(\phi_s|_{\langle a, b \rangle})\) and \(\Fix(\phi_s|_{\langle
		c\rangle})\).  The latter is immediately \(\langle c \rangle\), so it
		remains to compute the former. Let \(\psi_s = \phi_s|_{\langle a, b
		\rangle} \in \Aut(\langle a, b \rangle)\).

		To show \(\Fix(\psi_s) = \langle aba^{-1}, b \rangle\), we consider the
		free-by-cyclic group \(K = \langle a, b, s \mid sas^{-1} = ab, sbs^{-1}
		= b\rangle\), and note that by the same argument given above, \(C_K(s)
		= \langle s, \Fix(\psi_s)\rangle\). \cref{Bob-gp-cent}
		\eqref{item:L-cent-r} shows that \(C_K(s) = \langle s, aba^{-1}, b\rangle\),
		and so \(\Fix(\psi_s) = \langle aba^{-1}, b \rangle\), as required.

		To show \eqref{item:F2byF2-cent-b} and \eqref{item:F2byF2-cent-c}, let
		\(d \in \{b, c\}\). It is clear from the presentation that
		\(\langle d, s, t \rangle \subseteq C_G(d)\). For the reverse
		containment, let \(uv \in C_G(d)\), where \(u \in \langle a, b, c
		\rangle\) and \(v \in \langle s, t \rangle\). Then \(duv = uvd
		= udv\) (since \(d\) commutes with all elements in \(\langle s, t
		\rangle\)). Thus \(du = ud\); that is \(u \in C_{\langle a, b, c
		\rangle}(d) = \langle d \rangle\). We can conclude that
		\(C_G(d) = \langle d, s, t \rangle\), as required.

		For \eqref{item:F2byF2-cent-stk}, fix \(k \in \Z \setminus \{0\}\), and
		let \(uv \in C_G(st^k)\), with \(u \in \langle a, b, c \rangle\) and \(v
		\in \langle s, t \rangle\).  Then \(uvst^k = st^k uv = \phi_{st^k}(u)
		st^k v\). Then \(v \in C_{\langle s, t \rangle}(st^k)\), and so \(v =
		(st^k)^i\), for some \(i \in \Z\). Furthermore, if \(\phi_{st^k} \in
		\Aut(\langle a, b ,c \rangle)\) denotes the automorphism induced by
		conjugation by \(st^k\), then \(u = \phi_{st^k}(u)\); that is, \(u \in
		\Fix(\phi_{st^k})\).  We have that \(\phi_{st^k}(a) = b^k ab\), 
		\(\phi_{st^k}(b) = b\) and \(\phi_{st^k}(c) = c\). As with
		our proof of \eqref{item:F2byF2-cent-s}, noting that
		\(\phi_{st^k}\) fixes \(\langle a, b \rangle\) and \(\langle c \rangle\)
		setwise, it suffices to show that \(\Fix(\phi_{st^k}|_{\langle a, b \rangle})
		= \langle b \rangle\) and \(\Fix(\phi_{st^k}|_{\langle c \rangle}) = 
		\langle c \rangle\). The latter is immediate, so we consider
		the former.
		Let \(\psi_{st^k} = \psi_{st^k} |_{\langle a, b \rangle}\). We will use
		\cite[Theorem 2]{CohenLustig89} to show that \(\Fix(\psi_{st^k})\) is
		necessarily cyclic. We have that \(\psi_{st^k}(a) = b^k ab\), so for
		this to contain a primary occurrence of \(a\), it must be the only
		occurrence of \(a\). For this to be primary, we must have that \(b^k a
		b \equiv v^{-1} a w\) (that is, equivalent as words), with \(w =
		x^{-1}\cdot \psi_{st^k}(x)\). This implies that \(b = x^{-1} \cdot
		\psi_{st^k}(x)\), for some \(x \in \langle a, b \rangle\). But
		\(\expsum_b(x^{-1} \cdot \psi_{st^k}(x)) = - \expsum_b(x) +
		\expsum_b(x) + (k + 1)\expsum_a(x)\) and so \((k + 1)\expsum_a(x) =
		1\). If \(k \notin \{0, -2\}\), this is a contradiction. Thus
		\(\psi_{st^k}(a)\) contains no primary occurrences of \(a\). There is
		only one occurrence of \(b\) in \(\psi_{st^k}(b) = b\), and thus it
		follows by \cite[Theorem 2]{CohenLustig89}, that the rank of
		\(\Fix(\psi_{st^k})\) is at most \(0 + 1 = 1\). Since it contains the
		cyclic subgroup \(\langle b \rangle\) it must therefore by rank \(1\)
		and equal to \(\langle b \rangle\). We can conclude that \(uv \in
		\langle b, c, st^k \rangle\), and so \(C_G(st^k) \subseteq \langle b,c, st^k
		\rangle\) for all \(k \in \Z \setminus \{0, -2\}\).

It remains to consider the case \(k = -2\). For this case, we note that
		\(\psi_{st^k}(a) = b^{-2} ab\) and \(\psi_{st^k}(b) = b\). Consider a
		word \(u \in \Fix(\psi_{st^k})\). Write \(u \equiv b^{j_0}
		a^{\epsilon_1} b^{j_1} \cdots a^{\epsilon_n} b^{j_n}\), where \(j_\ell
		\in \Z\) and \(\epsilon_\ell \in \{-1, 1\}\) for all
		\(\ell\). For \(\epsilon \in \{-1, 1\}\), define
		\[
				\delta_\epsilon = \begin{cases}
						1 & \epsilon = 1 \\
						0 & \epsilon = -1
				\end{cases}
				\qquad
				\bar{\delta}_\epsilon = \begin{cases}
						0 & \epsilon = 1 \\
						1 & \epsilon = -1.
				\end{cases}
		\]
		Then 
		\[
				\psi_{st^k}(u) = b^{j_0 - (1 + \delta_{\epsilon_0})}
				a^{\epsilon_1} 
				b^{(1 + \bar{\delta}_{\epsilon_1}) + j_1 - (1 + \delta_{\epsilon_2})}
				a^{\epsilon_2} 
				\cdots
				b^{(1 + \bar{\delta}_{\epsilon_{n - 1}}) + j_{n - 1} - (1 + \delta_{\epsilon_n})}
				a^{\epsilon_n}
				b^{(1 + \bar{\delta}_{\epsilon_{n - 1}}) + j_n}
		\]
		So \(u \in \Fix(\psi_{st^k})\) implies that all of these exponents must
		coincide - none of the \(a\)s can cancel, as otherwise, \(\psi_{st^k}(u)\)
		would have fewer occurrences of \(a\)s and \(a^{-1}\)s than \(u\).
		In particular, we must have that \(j_0 = j_0 -(1 + \delta_{\epsilon_0})\),
		which implies that \(\delta_{\epsilon_0} = -1\), a contradiction, since
		the image of \(\delta\) is \(\{0, 1\}\). As a result, we must have no
		occurrences of \(a\) or \(a^{-1}\) in \(u\); that is, \(u \in \langle
		b \rangle\), as required. Noting that \(\langle b \rangle \subseteq
		\Fix(\psi_{st^k})\), it follows that \(\Fix(\psi_{st^k}) = \langle
		b \rangle\) in this case of \(k = -2\). Again, it follows that
		\(C_G(st^k) \subseteq \langle b,c,  st^k \rangle\).

		We have thus shown that for all \(k \in \Z \setminus \{0\}\), we have
		\(C_G(st^k) \subseteq \langle b,c, st^k \rangle\). The reverse
		containment is immediate from the presentation for \(G\).
	\end{proof}

	We next compute various definable subsets. When we embed Hilbert's tenth
	problem into \(G\), our copy of \(\Z\) will be \(\langle t \rangle\),
	although as will become clear during this proof, the exponents on each
	of the letters can be `transferred' to one another using equations.

	\begin{lem}
		\label{F2byF2-defable}
		Let \(G = \langle a, b, c, s, t \mid sas^{-1} = ab, sbs^{-1} = b,
		scs^{-1} = c, tat^{-1} = ba, tbt^{-1} = b, tct^{-1} = c \rangle\). 
		The following subsets of \(G\) are equationally definable:
		\begin{enumerate}
			\item \(\langle b \rangle\);
					\label{item:F2byF2-b-defable}
			\item \(\langle s \rangle\);
					\label{item:F2byF2-s-defable}
			\item \(\langle t \rangle\);
					\label{item:F2byF2-t-defable}
			\item \(\{(s^i, t^{i}) \mid i \in \Z\}\);
					\label{item:F2byF2-siti-defable}
			\item \(\langle s, t \rangle\).
					\label{item:F2byF2-st-defable}
		\end{enumerate}
	\end{lem}

	\begin{proof}
		For \eqref{item:F2byF2-b-defable}, note that \(\langle b, c \rangle
		= \langle aba^{-1}, b,c, s \rangle \cap \langle a^{-1}ba, b,c, t \rangle
		= C_G(s) \cap C_G(t)\), by \cref{F2byF2-cents}(1) and (2), and so is
		equationally definable. We then intersect this with \(C_G(b)\), noting
		that \(C_G(b) \cap \langle a, b, c \rangle = \langle b, c\rangle\) by
		\cref{F2byF2-cents}\eqref{item:F2byF2-cent-b},
		and so \(\langle b, c \rangle \cap C_G(b) = \langle b \rangle\)
		is equationally definable.

		For \eqref{item:F2byF2-s-defable}, let \(\mathcal E\) be a system of
		equations in \(G\) with variables \(X\) and \(Y\) the set of solutions
		to which is \(\langle b \rangle \times \langle aba^{-1}, b, c, s
		\rangle\), which exists by \eqref{item:F2byF2-b-defable} and
		\cref{F2byF2-cents}\eqref{item:F2byF2-cent-s}. Then consider the
		equation \(Ya Y^{-1} = a X\). Since
		\(s^i a s^{-i} = ab^i\) for all \(i \in \Z\), we have that the set of
		solutions to \((X, Y)\) in the system \(\mathcal E \wedge (Ya Y^{-1} =
		aX)\) contains the set \(\{(b^i, s^i) \mid i \in \Z\}\). Now let \((b^i,
		us^j)\) be a solution to \((X, Y)\) in the system \(\mathcal E \wedge
		(YaY^{-1} = aX)\), where \(u \in \langle aba^{-1}, b, c \rangle\). Then
		\begin{align*}
				ab^i = us^j a s^{-j} u^{-1} =
				uab^j u^{-1}.
		\end{align*}
		By equating \(b\)-exponent sums, it follows that \(i = j\). It thus
		follows that \(u \in C_{\langle a, b, c \rangle}(ab^i) = \langle ab^i
		\rangle\). But (for all possibilities of \(i\)) \(\langle ab^i \rangle
		\cap \langle aba^{-1}, b \rangle  =\{1\}\), and so \(u = 1\). We can
		conclude that the set of solutions to \((X, Y)\) in this system is
		\(\{(b^i, s^i) \mid i \in \Z\}\), and so \(\langle s \rangle\) is
		equationally definable.  \eqref{item:F2byF2-t-defable} follows by
		symmetry.

		We next consider \eqref{item:F2byF2-siti-defable}. Let \(\mathcal E\)
		be a system of equations with variables \(X\), \(Y\) and \(Z\), the set
		of solutions to which is \(\{(b^i, s^j, t^k) \mid i, j, k \in \Z\}\),
		which exists by \eqref{item:F2byF2-b-defable},
		\eqref{item:F2byF2-s-defable} and \eqref{item:F2byF2-t-defable}. We
		then add the equations \(aX = YaY^{-1}\) and \(Xa = Z a Z^{-1}\). Since
		\(s^i as^{-i} = ab^i\) and \(t^i a t^{-i} = b^{i} a\) for all \(i \in
		\Z\), the set of solutions to the system \(\mathcal E \wedge (aX =
		YaY^{-1}) \wedge (Z a Z^{-1})\) is \(\{(b^i, s^i, t^{i}) \mid i \in
		\Z\}\), as required.

		It remains to show \eqref{item:F2byF2-st-defable}. From
		\cref{F2byF2-cents}\eqref{item:F2byF2-cent-b} and
		\eqref{item:F2byF2-cent-c}, we have that \(C_G(b) = \langle b, s, t
		\rangle\) and \(C_G(c) = \langle c, s, t \rangle\). Thus
		\(\langle s, t \rangle = C_G(b) \cap C_G(c)\) is equationally
		definable.
	\end{proof}

	We can now prove the main result of this section - that \(G\) has an
	undecidable Diophantine problem. The `spirit' of the proof is similar
	to that in \cref{GBS-thm}, but many of the technical details differ;
	most notably that we embed Hilbert's tenth problem over the ring of
	integers, rather than the natural numbers. The central idea of
	the proof is to use the action of \(\langle s, t \rangle\) on
	\(\langle a, b \rangle\) to `count' the exponent sums of
	\(s\) and \(t\) in a given element of \(\langle s, t \rangle\).

	\begin{theorem}\label{thm:FreeByFree}
		\label{F2byF2-thm}
	There is a group of the form $F_3 \rtimes F_2$ with an undecidable Diophantine problem. In particular the group
\[
\langle a, b,c, s, t \mid sas^{-1} = ab, sbs^{-1} = b, scs^{-1} = c, tat^{-1} = ba, tbt^{-1} = b, tct^{-1} = c \rangle
\]
has an undecidable Diophantine problem.	
	\end{theorem}

\begin{proof}
		Similar to the argument in \cref{GBS-thm}, we reduce Hilbert's tenth
		problem in the ring of integers to the Diophantine problem for \(G\),
		albeit this time we directly embed Hilbert's tenth problem for the ring
		of integers, rather than the analogue for the natural numbers. Since
		Hilbert's tenth problem over the ring of integers is undecidable
		\cite{Matijasevic}, the result will then follow. The copy of \(\Z\)
		inside \(G\) we will be using will be \(\langle t \rangle\), which is
		equationally definable, by
		\cref{F2byF2-defable}\eqref{item:F2byF2-t-defable}. To encode Hilbert's
		tenth problem in the ring of integers into \(G\), it suffices to encode
		addition and multiplication; this is, show that the sets \(\{(t^i, t^j,
		t^{i + j}) \mid i, j \in \Z\}\) and \(\{(t^i, t^j, t^{ij}) \mid i, j
\in \Z\}\) are equationally definable.  The former is straightforward; we
simply take a system \(\mathcal E\) of equations with variables \(X\) and \(Y\),
the set of solutions to which is \(\{(t^i, t^j) \mid i, j \in \Z\}\), which
exists by \cref{F2byF2-defable}\eqref{item:F2byF2-t-defable}, and then consider
the system \(\mathcal E \wedge (Z = X Y)\), where \(Z\) is a new variable. The
set of solutions to \((X, Y, Z)\) in this system is \(\{(t^i, t^j, t^{i + j})
\mid i, j \in \Z\}\), as required.

We next consider multiplication. First let \(\mathcal E\) be a system
		of equations in \(G\), with variables \(X\), \(Y\), \(Z\) and \(W\),
		where the set of solutions to \((W,  X, Y, Z)\) in \(\mathcal
		E\) is \(\{(s^j, t^i, t^j, t^k) \mid i, j, k \in \Z, i
		\textrm{ odd}\}\), which exists by \cref{F2byF2-defable}\eqref{item:F2byF2-s-defable}
		and \eqref{item:F2byF2-t-defable}
(we obtain
		\(i\) odd from all \(i\) by defining \(X = (X')^2 t\), where \(X'\) is
		a variable taking values in \(\langle t \rangle\)). We will add
		additional equations one at a time, and then compute the solution set
		to the system defined at each point, until we have one showing that
		\(\{(t^i, t^j, t^{ij}) \mid i, j \in \Z, i \textrm{ odd}\}\) is
		equationally definable; we discuss the difference between the desired
		set and this set once we have shown it to be equationally definable.
		We first add the equations \(V sX = sX V\), where \(V\) is a new
		variable. Since this simply says \(V\) centralises \(sX\),
		which by \cref{F2byF2-cents}\eqref{item:F2byF2-cent-stk} has solutions \(\{b^q (st^i)^m \mid m,
		q\in \Z\}\)
 (noting \(i \neq 0\) since \(i\) is always odd)
		we can conclude that the set of solutions to \((V, W, X, Y, Z)\)
		is
		\begin{align*}
				&\{(b^q (st^i)^m, s^j, t^i, t^j, t^k) \mid i, j, k, m,q \in \Z, i \textrm{ odd}\}.
		\end{align*}
		We next add the constraint \(V \in \langle s, t \rangle\), which
		is equationally definable by \cref{F2byF2-defable}\eqref{item:F2byF2-st-defable}.
		Using the above notation, the solutions to \(V\) are the elements
		among \(b^q (st^i)^m\) that lie in \(\langle s, t \rangle\); that
		is, when \(q = 0\). We have thus shown that the set
		\begin{align*}
				&\{((st^i)^m, s^j, t^i, t^j, t^k) \mid i, j, k, m \in \Z, i \textrm{ odd}\}.
		\end{align*}
		is equationally definable.

		We next add the equation \(VaV^{-1} = WZ a (WZ)^{-1}\). Using the notation
		above, this can be rewritten as \((st^i)^m a (st^i)^{-m} = (s^j t^k) a
		(s^j t^k)^{-1}\). We have that \((st^i)^m a (st^i)^{-m} = b^{im} a b^m\)
		and \((s^j t^k) a (s^j t^k)^{-1} = b^k a b^j\). We can thus
		conclude that \(m = j\) and \(k = im\), thus making our solution set
		thus far equal to
		\begin{align*}
				&\{((st^i)^j, s^j, t^i, t^j, t^{ij}) \mid i, j \in \Z, i \textrm{ odd}\}.
		\end{align*}

		By looking at the variables \((X, Y, Z)\), we have shown that the set
		\(\{(t^i, t^j, t^{ij}) \mid i, j \in \Z, i \textrm{ odd}\}\) is
		equationally definable in \(G\).  By adding a new variable \(P\) along
		with the equation \(P = Xt\), we also obtain that the set \(\{(t^{2i'},
		t^j, t^{2i'j}) \mid i', j \in \Z\}\) is equationally definable, where
		here \(2i' = i + 1\) (the solutions to \((P, Y, W')\)).  Thus by
		choosing new variables \(Q\) and \(R\) and adding the equations \(Q^2 =
		P\) and \(R^2 = W'\), we obtain that the set of solutions to \((Q, Y,
		R)\) is \(\{(t^{i''}, t^j, t^{i''j}) \mid i'', j \in \Z\}\), as
		required.
\end{proof}

\bibliographystyle{custom} 
\bibliography{references}

\begin{thebibliography}{10}

\bibitem{mathoverflow}
https://mathoverflow.net/questions/160161/infinitely-many-finitely-generated-groups-having-the-same-cayley-graph/160178\#160178.

\bibitem{Levent2026}
L.~Alp\"oge, M.~Bhargava, W.~Ho, and A.~Shnidman.
\newblock Rank stability in quadratic extensions and {H}ilbert's tenth problem for the ring of integers of a number field.
\newblock {\em Invent. Math.}, 243(3):1129--1139, 2026.

\bibitem{AntolinFoniqi22}
Y.~Antol\'in and I.~Foniqi.
\newblock Intersection of parabolic subgroups in even {A}rtin groups of {FC}-type.
\newblock {\em Proc. Edinb. Math. Soc. (2)}, 65(4):938--957, 2022.

\bibitem{AscariCasalsRuizKazachkov}
D.~Ascari, M.~Casals-Ruiz, and I.~Kazachkov.
\newblock {On the isomorphism problem for generalized Baumslag-Solitar groups: invariants and flexible configurations}, 2025.
\newblock arXiv:2508.02306.

\bibitem{BarkThesis}
D.~A. Barkauskas.
\newblock {\em Centralizers in Fundamental Groups of Graphs of Groups}.
\newblock PhD thesis, 2003.

\bibitem{BaumslagMorganShalen}
G.~Baumslag, J.~W. Morgan, and P.~B. Shalen.
\newblock Generalized triangle groups.
\newblock {\em Math. Proc. Cambridge Philos. Soc.}, 102(1):25--31, 1987.

\bibitem{baumslag2002open}
G.~Baumslag, A.~G. Myasnikov, and V.~Shpilrain.
\newblock Open problems in combinatorial group theory.
\newblock {\em Contemporary Mathematics}, 296:1--38, 2002.

\bibitem{Bogopolski2006}
O.~Bogopolski, A.~Martino, O.~Maslakova, and E.~Ventura.
\newblock The conjugacy problem is solvable in free-by-cyclic groups.
\newblock {\em Bull. London Math. Soc.}, 38(5):787--794, 2006.

\bibitem{BogopolskiMartinoVentura}
O.~Bogopolski, A.~Martino, and E.~Ventura.
\newblock Orbit decidability and the conjugacy problem for some extensions of groups.
\newblock {\em Trans. Amer. Math. Soc.}, 362(4):2003--2036, 2010.

\bibitem{BogopolskiBook}
O.~Bogopolski.
\newblock {\em Introduction to group theory}.
\newblock EMS Textbooks in Mathematics. European Mathematical Society (EMS), Z\"urich, 2008.
\newblock Translated, revised and expanded from the 2002 Russian original.

\bibitem{Borisov1969}
V.~V. Borisov.
\newblock Simple examples of groups with unsolvable word problem.
\newblock {\em Mat. Zametki}, 6:521--532, 1969.

\bibitem{bridson2025conjugacy}
M.~R. Bridson, T.~R. Riley, and A.~W. Sale.
\newblock Conjugacy in a family of free-by-cyclic groups.
\newblock {\em arXiv preprint arXiv:2506.01248}, 2025.

\bibitem{Brodskii84}
S.~D. Brodski\u{\i}.
\newblock Equations over groups, and groups with one defining relation.
\newblock {\em Sibirsk. Mat. Zh.}, 25(2):84--103, 1984.

\bibitem{BuchiSenger}
J.~R. B\"uchi and S.~Senger.
\newblock Definability in the existential theory of concatenation and undecidable extensions of this theory.
\newblock {\em Z. Math. Logik Grundlag. Math.}, 34(4):337--342, 1988.

\bibitem{BurilloElder}
J.~Burillo and M.~Elder.
\newblock Metric properties of {B}aumslag-{S}olitar groups.
\newblock {\em Internat. J. Algebra Comput.}, 25(5):799--811, 2015.

\bibitem{Burns1987}
R.~G. Burns, A.~Karrass, and D.~Solitar.
\newblock A note on groups with separable finitely generated subgroups.
\newblock {\em Bull. Austral. Math. Soc.}, 36(1):153--160, 1987.

\bibitem{Cadilhac2020}
M.~Cadilhac, D.~Chistikov, and G.~Zetzsche.
\newblock {Rational Subsets of Baumslag-Solitar Groups}.
\newblock In A.~Czumaj, A.~Dawar, and E.~Merelli, editors, {\em 47th International Colloquium on Automata, Languages, and Programming (ICALP 2020)}, volume 168 of {\em Leibniz International Proceedings in Informatics (LIPIcs)}, pages 116:1--116:16, Dagstuhl, Germany, 2020. Schloss Dagstuhl -- Leibniz-Zentrum f{\"u}r Informatik.

\bibitem{CampbellHeggieRobertsonThomas}
C.~M. Campbell, P.~M. Heggie, E.~F. Robertson, and R.~M. Thomas.
\newblock Cyclically presented groups embedded in one-relator products of cyclic groups.
\newblock {\em Proc. Amer. Math. Soc.}, 118(2):401--408, 1993.

\bibitem{CasalsRuiz2011}
M.~Casals-Ruiz and I.~Kazachkov.
\newblock On systems of equations over free partially commutative groups.
\newblock {\em Mem. Amer. Math. Soc.}, 212(999):viii+153, 2011.

\bibitem{CasalsRuiz2024}
M.~Casals-Ruiz, I.~Kazachkov, and J.~de~la Nuez~Gonz\'alez.
\newblock On the elementary theory of graph products of groups.
\newblock {\em Dissertationes Math.}, 595:92, 2024.

\bibitem{CasalsRuizetal21}
M.~Casals-Ruiz, I.~Kazachkov, and A.~Zakharov.
\newblock Commensurability of {B}aumslag-{S}olitar groups.
\newblock {\em Indiana Univ. Math. J.}, 70(6):2527--2555, 2021.

\bibitem{CiobanuEvettsLevine}
L.~Ciobanu, A.~Evetts, and A.~Levine.
\newblock Effective equation solving, constraints, and growth in virtually abelian groups.
\newblock {\em SIAM J. Appl. Algebra Geom.}, 9(1):235--260, 2025.

\bibitem{CiobanuGarreta}
L.~Ciobanu and A.~Garreta.
\newblock Group equations with abelian predicates.
\newblock {\em Int. Math. Res. Not. IMRN}, (5):4119--4159, 2024.

\bibitem{CiobanuGrayLevineInPrep}
L.~Ciobanu, R.~D. Gray, and A.~Levine.
\newblock The {D}iophantine problem with abelianisation constraints in some free-by-cyclic groups.
\newblock {\em In preparation.}

\bibitem{CiobanuHoltRees}
L.~Ciobanu, D.~Holt, and S.~Rees.
\newblock Equations in groups that are virtually direct products.
\newblock {\em J. Algebra}, 545:88--99, 2020.

\bibitem{CohenLustig89}
M.~M. Cohen and M.~Lustig.
\newblock On the dynamics and the fixed subgroup of a free group automorphism.
\newblock {\em Invent. Math.}, 96(3):613--638, 1989.

\bibitem{Collins1977}
D.~J. Collins.
\newblock Generation and presentation of one-relator groups with centre.
\newblock {\em Mathematische Zeitschrift}, 157(1):63--77, 1977.

\bibitem{Dahmani}
F.~Dahmani.
\newblock Existential questions in (relatively) hyperbolic groups.
\newblock {\em Israel J. Math.}, 173:91--124, 2009.

\bibitem{dahmani_guirardel}
F.~Dahmani and V.~Guirardel.
\newblock Foliations for solving equations in groups: free, virtually free, and hyperbolic groups.
\newblock {\em J. Topol.}, 3(2):343--404, 2010.

\bibitem{dahmani2009existential}
F.~Dahmani.
\newblock Existential questions in (relatively) hyperbolic groups.
\newblock {\em Israel Journal of Mathematics}, 173:91--124, 2009.

\bibitem{DahmaniRuoyu}
F.~Dahmani and R.~Li.
\newblock Relative hyperbolicity for automorphisms of free products and free groups.
\newblock {\em J. Topol. Anal.}, 14(1):55--92, 2022.

\bibitem{DayWade19}
M.~B. Day and R.~D. Wade.
\newblock Relative automorphism groups of right-angled {A}rtin groups.
\newblock {\em J. Topol.}, 12(3):759--798, 2019.

\bibitem{delaHarpeBook}
P.~de~la Harpe.
\newblock {\em Topics in geometric group theory}.
\newblock Chicago Lectures in Mathematics. University of Chicago Press, Chicago, IL, 2000.

\bibitem{DiekertMuscholl}
V.~Diekert and A.~Muscholl.
\newblock Solvability of equations in free partially commutative groups is decidable.
\newblock In {\em Automata, languages and programming}, volume 2076 of {\em Lecture Notes in Comput. Sci.}, pages 543--554. Springer, Berlin, 2001.

\bibitem{HydraGroups}
W.~Dison and T.~R. Riley.
\newblock Hydra groups.
\newblock {\em Comment. Math. Helv.}, 88(3):507--540, 2013.

\bibitem{Dong25}
R.~Dong.
\newblock Linear equations with monomial constraints and decision problems in abelian-by-cyclic groups.
\newblock In {\em Proceedings of the 2025 {A}nnual {ACM}-{SIAM} {S}ymposium on {D}iscrete {A}lgorithms ({SODA})}, pages 1892--1908. SIAM, Philadelphia, PA, 2025.

\bibitem{duchin_liang_shapiro}
M.~Duchin, H.~Liang, and M.~Shapiro.
\newblock Equations in nilpotent groups.
\newblock {\em Proc. Amer. Math. Soc.}, 143(11):4723--4731, 2015.

\bibitem{Duncan2023}
A.~Duncan, A.~Evetts, D.~F. Holt, and S.~Rees.
\newblock Using {EDT}0{L} systems to solve some equations in the solvable {B}aumslag-{S}olitar groups.
\newblock {\em J. Algebra}, 630:434--456, 2023.

\bibitem{Duncan1991}
A.~J. Duncan and J.~Howie.
\newblock The genus problem for one-relator products of locally indicable groups.
\newblock {\em Math. Z.}, 208(2):225--237, 1991.

\bibitem{DuncanHowie92}
A.~J. Duncan and J.~Howie.
\newblock Weinbaum's conjecture on unique subwords of nonperiodic words.
\newblock {\em Proc. Amer. Math. Soc.}, 115(4):947--954, 1992.

\bibitem{DuncanHowie93}
A.~J. Duncan and J.~Howie.
\newblock One relator products with high-powered relators.
\newblock In {\em Geometric group theory, {V}ol.\ 1 ({S}ussex, 1991)}, volume 181 of {\em London Math. Soc. Lecture Note Ser.}, pages 48--74. Cambridge Univ. Press, Cambridge, 1993.

\bibitem{DuncanArye}
A.~J. Duncan and A.~Juh\'asz.
\newblock One-relator quotients of right-angled {A}rtin groups.
\newblock {\em J. Algebra}, 622:506--555, 2023.

\bibitem{ElliottLevine}
L.~Elliott and A.~Levine.
\newblock The {D}iophantine problem in {T}hompson's group {$F$}.
\newblock {\em Math. Comp.}, 95(360):2013--2023, 2026.

\bibitem{Ersov72}
J.~L. Er\v{s}ov.
\newblock Elementary group theories.
\newblock {\em Dokl. Akad. Nauk SSSR}, 203:1240--1243, 1972.

\bibitem{FineHowieRosenberger88}
B.~Fine, J.~Howie, and G.~Rosenberger.
\newblock One-relator quotients and free products of cyclics.
\newblock {\em Proc. Amer. Math. Soc.}, 102(2):249--254, 1988.

\bibitem{FineRosenberger}
B.~Fine and G.~Rosenberger.
\newblock {\em Algebraic generalizations of discrete groups}, volume 223 of {\em Monographs and Textbooks in Pure and Applied Mathematics}.
\newblock Marcel Dekker, Inc., New York, 1999.
\newblock A path to combinatorial group theory through one-relator products.

\bibitem{FriedlWilton}
S.~Friedl and H.~Wilton.
\newblock The membership problem for 3-manifold groups is solvable.
\newblock {\em Algebr. Geom. Topol.}, 16(4):1827--1850, 2016.

\bibitem{Garreta2021}
A.~Garreta and R.~D. Gray.
\newblock On equations and first-order theory of one-relator monoids.
\newblock {\em Inform. and Comput.}, 281:Paper No. 104745, 19, 2021.

\bibitem{random_nilpotent_eqns}
A.~Garreta, A.~Miasnikov, and D.~Ovchinnikov.
\newblock Random nilpotent groups, polycyclic presentations, and {D}iophantine problems.
\newblock {\em Groups Complex. Cryptol.}, 9(2):99--115, 2017.

\bibitem{Garreta2020}
A.~Garreta, A.~Miasnikov, and D.~Ovchinnikov.
\newblock Diophantine problems in solvable groups.
\newblock {\em Bull. Math. Sci.}, 10(1):2050005, 27, 2020.

\bibitem{gautero2007mapping}
F.~Gautero and M.~Lustig.
\newblock The mapping-torus of a free group automorphism is hyperbolic relative to the canonical subgroups of polynomial growth.
\newblock {\em arXiv preprint arXiv:0707.0822}, 2007.

\bibitem{Ghosh23}
P.~Ghosh.
\newblock Relative hyperbolicity of free-by-cyclic extensions.
\newblock {\em Compos. Math.}, 159(1):153--183, 2023.

\bibitem{GrayLevineArxiv251207690}
R.~D. Gray and A.~Levine.
\newblock Decidability of equations and first-order theory in {S}eifert 3-manifold groups.
\newblock {\em arXiv preprint arXiv:2512.07690}, 2025.

\bibitem{Gromov1987}
M.~Gromov.
\newblock Hyperbolic groups.
\newblock In {\em Essays in group theory}, volume~8 of {\em Math. Sci. Res. Inst. Publ.}, pages 75--263. Springer, New York, 1987.

\bibitem{Howie81}
J.~Howie.
\newblock On pairs of {$2$}-complexes and systems of equations over groups.
\newblock {\em J. Reine Angew. Math.}, 324:165--174, 1981.

\bibitem{Howie84}
J.~Howie.
\newblock Cohomology of one-relator products of locally indicable groups.
\newblock {\em J. London Math. Soc. (2)}, 30(3):419--430, 1984.

\bibitem{HowieFourthPowers}
J.~Howie.
\newblock The quotient of a free product of groups by a single high-powered relator. {II}.\ {F}ourth powers.
\newblock {\em Proc. London Math. Soc. (3)}, 61(1):33--62, 1990.

\bibitem{JaikinZapirainLintonSanchezPeralta}
A.~Jaikin-Zapirain, M.~Linton, and P.~Sánchez-Peralta.
\newblock Group pairs, coherence and {F}arrell--{J}ones conjecture for {$K_0$}.
\newblock {\em arXiv preprint arXiv:2510.23518}, 2025.

\bibitem{Kharlampovich2019}
O.~Kharlampovich and L.~L\'opez.
\newblock Bi-interpretability of some monoids with the arithmetic and applications.
\newblock {\em Semigroup Forum}, 99(1):126--139, 2019.

\bibitem{kharlampovich2020diophantine}
O.~Kharlampovich, L.~L{\'o}pez, and A.~Myasnikov.
\newblock The {D}iophantine problem in some metabelian groups.
\newblock {\em Mathematics of computation}, 89(325):2507--2519. Corrected version: arxiv.org/abs/1903.10068, 2020.

\bibitem{KharlampovichMiasnikov2025}
O.~Kharlampovich and A.~Miasnikov.
\newblock The {D}iophantine problem in iterated wreath products of free abelian groups is undecidable.
\newblock {\em arXiv e-prints}, arXiv:2502.09442v1, 2025.

\bibitem{koymans2024hilbert}
P.~Koymans and C.~Pagano.
\newblock Hilbert's tenth problem via additive combinatorics.
\newblock {\em arXiv preprint arXiv:2412.01768}, 2024.

\bibitem{Kropholler1990}
P.~H. Kropholler.
\newblock Baumslag-{S}olitar groups and some other groups of cohomological dimension two.
\newblock {\em Commentarii Mathematici Helvetici}, 65(1):547--558, 1990.

\bibitem{Kropholler1993}
P.~H. Kropholler.
\newblock A group theoretic proof of the torus theorem.
\newblock {\em Geometric group theory}, 1:138--158, 1993.

\bibitem{kudlinska2024subgroup}
M.~Kudlinska.
\newblock On subgroup separability of free-by-cyclic and deficiency 1 groups.
\newblock {\em Bulletin of the London Mathematical Society}, 56(1):338--351, 2024.

\bibitem{Levine2022}
A.~Levine.
\newblock Equations in virtually class 2 nilpotent groups.
\newblock {\em J. Groups Complex. Cryptol.}, 14(1):Paper No. 2, 17, 2022.

\bibitem{levitt2015generalized}
G.~Levitt.
\newblock Generalized {B}aumslag--{S}olitar groups: rank and finite index subgroups.
\newblock In {\em Annales de l'Institut Fourier}, volume~65, pages 725--762, 2015.

\bibitem{linton2025geometry}
M.~Linton.
\newblock The geometry of subgroups of mapping tori of free groups.
\newblock {\em arXiv preprint arXiv:2510.03145}, 2025.

\bibitem{CFMarcoOneRelatorSurvey}
M.~Linton and C.-F. Nyberg-Brodda.
\newblock The theory of one-relator groups: history and recent progress.
\newblock {\em arXiv preprint arXiv:2501.18306}, 2025.

\bibitem{Lipschutz1966}
S.~Lipschutz.
\newblock Generalization of {D}ehn's result on the conjugacy problem.
\newblock {\em Proc. Amer. Math. Soc.}, 17:759--762, 1966.

\bibitem{Logan}
A.~D. Logan.
\newblock The conjugacy problem for ascending {HNN}-extensions of free groups.
\newblock {\em Proc. Lond. Math. Soc. (3)}, 131(4):Paper No. e70088, 24, 2025.

\bibitem{LohreySenizergues}
M.~Lohrey and G.~S\'{e}nizergues.
\newblock Theories of {HNN}-extensions and amalgamated products.
\newblock In {\em Automata, languages and programming. {P}art {II}}, volume 4052 of {\em Lecture Notes in Comput. Sci.}, pages 504--515. Springer, Berlin, 2006.

\bibitem{lyndon1977combinatorial}
R.~C. Lyndon, P.~E. Schupp, R.~Lyndon, and P.~Schupp.
\newblock {\em Combinatorial group theory}, volume~89.
\newblock Springer, 1977.

\bibitem{magnus2004combinatorial}
W.~Magnus, A.~Karrass, and D.~Solitar.
\newblock {\em Combinatorial group theory: Presentations of groups in terms of generators and relations}.
\newblock Courier Corporation, 2004.

\bibitem{Makanin_semigroups}
G.~S. Makanin.
\newblock The problem of the solvability of equations in a free semigroup.
\newblock {\em Mat. Sb. (N.S.)}, 103(145)(2):147--236, 319, 1977.

\bibitem{Makanin_eqns_free_group}
G.~S. Makanin.
\newblock Equations in a free group.
\newblock {\em Izv. Akad. Nauk SSSR Ser. Mat.}, 46(6):1199--1273, 1344, 1982.

\bibitem{Mandel2023}
R.~Mandel and A.~Ushakov.
\newblock The {D}iophantine problem for systems of algebraic equations with exponents.
\newblock {\em J. Algebra}, 636:779--803, 2023.

\bibitem{Mandel2023b}
R.~Mandel and A.~Ushakov.
\newblock Quadratic equations in metabelian {B}aumslag-{S}olitar groups.
\newblock {\em Internat. J. Algebra Comput.}, 33(6):1195--1216, 2023.

\bibitem{MartinoVentura}
A.~Martino and E.~Ventura.
\newblock A description of auto-fixed subgroups in a free group.
\newblock {\em Topology}, 43(5):1133--1164, 2004.

\bibitem{Matijasevic}
J.~V. Matijasevi\v{c}.
\newblock The {D}iophantineness of enumerable sets.
\newblock {\em Dokl. Akad. Nauk SSSR}, 191:279--282, 1970.

\bibitem{McCool1991}
J.~McCool.
\newblock A class of one-relator groups with centre.
\newblock {\em Bulletin of the Australian Mathematical Society}, 44(2):245--252, 1991.

\bibitem{MillerBook}
C.~F. Miller, III.
\newblock {\em On group-theoretic decision problems and their classification}, volume No. 68 of {\em Annals of Mathematics Studies}.
\newblock Princeton University Press, Princeton, NJ; University of Tokyo Press, Tokyo, 1971.

\bibitem{Minasyan2013}
A.~Minasyan and P.~Zalesskii.
\newblock One-relator groups with torsion are conjugacy separable.
\newblock {\em J. Algebra}, 382:39--45, 2013.

\bibitem{MuschollThesis}
A.~Muscholl.
\newblock Decision and complexity issues on concurrent systems, 1999.
\newblock Habilitation thesis.

\bibitem{noskov1984elementary}
G.~A. Noskov.
\newblock On the elementary theory of a finitely generated almost solvable group.
\newblock {\em Mathematics of the USSR-Izvestiya}, 22(3):465, 1984.

\bibitem{CFOneRelatorMonoidsSuvey}
C.-F. Nyberg-Brodda.
\newblock The word problem for one-relation monoids: a survey.
\newblock {\em Semigroup Forum}, 103(2):297--355, 2021.

\bibitem{CFOneRelatorDiophantine}
C.-F. Nyberg-Brodda.
\newblock On the {D}iophantine problem in some one-relator groups.
\newblock {\em arXiv e-prints}, arXiv:2208.07145, 2022.

\bibitem{Pietrowski1974}
A.~Pietrowski.
\newblock The isomorphism problem for one-relator groups with non-trivial centre.
\newblock {\em Mathematische Zeitschrift}, 136(2):95--106, 1974.

\bibitem{Pride2008}
S.~J. Pride.
\newblock On the residual finiteness and other properties of (relative) one-relator groups.
\newblock {\em Proc. Amer. Math. Soc.}, 136(2):377--386, 2008.

\bibitem{puder2014primitive}
D.~Puder.
\newblock Primitive words, free factors and measure preservation.
\newblock {\em Israel Journal of Mathematics}, 201(1):25--73, 2014.

\bibitem{Razborov_thesis}
A.~A. Razborov.
\newblock {\em On systems of equations in free groups}.
\newblock PhD thesis, Steklov Institute of Mathematics, 1971.
\newblock In Russian.

\bibitem{Sela}
E.~Rips and Z.~Sela.
\newblock Canonical representatives and equations in hyperbolic groups.
\newblock {\em Invent. Math.}, 120(3):489--512, 1995.

\bibitem{romanovskiui1981elementary}
N.~S. Romanovski\u{\i}.
\newblock The elementary theory of an almost polycyclic group.
\newblock {\em Mat. Sb. (N.S.)}, 111(153)(1):135--143, 160, 1980.

\bibitem{Servatius89}
H.~Servatius.
\newblock Automorphisms of graph groups.
\newblock {\em J. Algebra}, 126(1):34--60, 1989.

\bibitem{ShortThesis}
H.~Short.
\newblock {\em Topological methods in group theory: the adjunction problem}.
\newblock PhD thesis, University of Warwick, 1983.

\bibitem{Vogtmann15}
K.~Vogtmann.
\newblock {$GL(n,\Bbb Z)$}, {$Out(F_n)$} and everything in between: automorphism groups of {RAAG}s.
\newblock In {\em Groups {S}t {A}ndrews 2013}, volume 422 of {\em London Math. Soc. Lecture Note Ser.}, pages 105--127. Cambridge Univ. Press, Cambridge, 2015.

\bibitem{WeissThesis}
A.~Wei{\ss}.
\newblock {\em On the complexity of conjugacy in amalgamated products and HNN extensions}.
\newblock PhD thesis, Universit\:{a}t Stuttgart, 2015.

\bibitem{Whyte01}
K.~Whyte.
\newblock The large scale geometry of the higher {B}aumslag-{S}olitar groups.
\newblock {\em Geom. Funct. Anal.}, 11(6):1327--1343, 2001.

\end{thebibliography}

\end{document}